\documentclass{amsart}
\usepackage{amsfonts, amsthm, amsmath, enumerate, mathrsfs, tikz, tikz-cd, ifthen, color, chngcntr, mathtools, multicol,array, bbm, euscript}

\usepackage[colorlinks]{hyperref}
	\hypersetup{linkcolor={blue}, citecolor={black}, urlcolor={blue}}
\usepackage{subcaption}
\usepackage[toc,page]{appendix}
\usetikzlibrary{arrows}
\usetikzlibrary{decorations.markings}

\newtheorem{theorem}[subsection]{Theorem}
\newtheorem{prop}[subsection]{Proposition}
\newtheorem{lemma}[subsection]{Lemma}

\newtheorem{corollary}[subsection]{Corollary}
\theoremstyle{definition}
\newtheorem{definition}[subsection]{Definition}
\newtheorem{example}[subsection]{Example}
\newtheorem{notation}[subsection]{Notation}
\newtheorem{rmk}[subsection]{Remark}
\definecolor{light-gray}{gray}{.65}

\newcommand{\cat}{\EuScript} 
\newcommand{\Top}{{\cat Top}}
\newcommand{\M}{\mathbb{M}}
\newcommand{\R}{\mathbb{R}}
\newcommand{\Z}{\mathbb{Z}}

\renewcommand{\H}{\tilde{H}}
\renewcommand{\dot}{\small{$\bullet$}}

\DeclareMathOperator{\cof}{cof}
\DeclareMathOperator{\coker}{coker}
\DeclareMathOperator{\im}{im}

\newcommand{\cone}[3]{
	\draw[thick, #3] (#1+1/2,5) -- (#1+1/2,#2+1/2) -- (5-#2+#1,5);
	\draw[thick, #3] (#1+1/2,-5) -- (#1 +1/2, #2 -1.5) -- (-3-#2+#1,-5)
}
\newcommand{\anti}[3]{
	\draw[thick, #3] (#1+1/2, -5) -- (#1+1/2, 5);
	\draw[thick, #3] (#1+#2+1/2, -5) -- (#1+#2+1/2, 5);
	{\ifthenelse{#2>1}{
	\foreach \y in {-5,...,2} {
	\draw[thick, #3] (#1+1/2, \y +1/2) -- (#1+#2+1/2, \y + #2+1/2);
	}
	}{
	\foreach \y in {-5,-4,-3,-2,-1,0,1,2,3} {\draw[thick, #3] (#1+1/2, \y +1/2) -- (#1+#2+1/2, \y + #2+1/2);}}
	};
}
\newcommand{\lab}[3]{
\draw[#3] (#1+1/2,5.5) node{\tiny{$#2$}}
}

\counterwithin{equation}{subsection}
\begin{document}

\title{The $RO(C_2)$-graded cohomology of $C_2$-Surfaces in $\underline{\Z/2}$-coefficients}



\begin{abstract} 
A surface with an involution can be viewed as a $C_2$-space where $C_2$ is the cyclic group of order two. Up to equivariant isomorphism, all involutions on surfaces were classified in \cite{BCNS} and recently classified using equivariant surgery in \cite{D2}. We use the classification given in \cite{D2} to compute the $RO(C_2)$-graded Bredon cohomology of all $C_2$-surfaces in constant $\Z/2$ coefficients as modules over the cohomology of a point. We show the cohomology depends only on three numerical invariants in the nonfree case, and only on two numerical invariants in the free case.  
\end{abstract}

\author[C. Hazel]{Christy Hazel}
\maketitle

\tableofcontents

\section{Introduction} 

There has been recent interest in $RO(C_2)$-graded cohomology computations, and a number of concrete computations now exist in the literature, such as \cite{dS}, \cite{Ho}, \cite{K2}, \cite{LFdS1}, \cite{LFdS2}, and \cite{S}. Computations have also been done in an extended grading given by the fundamental groupoid, such as in \cite{CHT}. The goal of this paper is to compute the $RO(C_2)$-graded cohomology of all $C_2$-surfaces in constant $\Z/2$ coefficients, and to state the answer in a concise and coherent way based on a few properties of the space and its action. We show the answers are entirely dependent on three invariants of the $C_2$-surface in the nonfree case, and two invariants in the free case. Even better, the cohomology can be formulaically described in terms of these invariants.  In the appendix of this paper, we prove the existence of a top cohomology class in $\underline{\Z/2}$-coefficients for nonfree $C_2$-manifolds of any dimension. \smallskip

Let us begin by recalling some facts about Bredon cohomology. For a finite group $G$ the Bredon cohomology of a $G$-space is a sequence of abelian groups graded on $RO(G)$, the Grothendieck group of finite-dimensional, real, orthogonal $G$-representations. When $G$ is the cyclic group of order two, recall any $C_2$-representation is isomorphic to a direct sum of trivial representations and sign representations. Thus $RO(C_2)$ is a free abelian group of rank two, and the Bredon cohomology of any $C_2$-space can be regarded as a collection of bigraded abelian groups. 

These bigraded groups fit together to form a bigraded algebra over the cohomology of a point. We will use the notation $H^{*,*}(X;M)$ to denote the Bredon cohomology of a $C_2$-space $X$ with coefficients in $M$. We will often refer to the first grading as the ``topological dimension" and the second grading as the ``weight" (the specific conventions are explained in Section \ref{ch:preliminaries}). Loosely speaking, we can think of the first dimension as measuring something about the underlying topological space, while the second grading measures something about the nontriviality of the action. 

When working in \underline{$\Z/2$}-coefficients, it is shown in \cite{M1} that, as a module over the cohomology of a point, the cohomology of any finite $C_2$-CW complex decomposes into a direct sum of two types of summands. Specifically, the cohomology is given by a direct sum of free modules and shifted copies of the cohomology of antipodal spheres. We provide an introduction to these two types of modules below. 

\subsection{The Basic Pieces} 
Since the cohomology of any $C_2$-space is a bigraded module over a bigraded ring, we can use a grid to record information about the cohomology groups and module structures. For example, the cohomology ring of a point in $\underline{\Z/2}$-coefficients is illustrated on the left-hand grid in the figure below. Each dot represents a copy of $\Z/2$, and the connecting lines indicate properties of the ring structure. For example, the top portion is polynomial in two elements $\rho$ and $\tau$ which are in bidegrees $(1,1)$ and $(0,1)$, respectively. A full description of this ring can be found in Section \ref{ch:preliminaries}, but for now we just provide a picture to give the reader an idea. In practice, it is cumbersome to draw the detailed picture, so instead, we draw the abbreviated version shown on the right. 

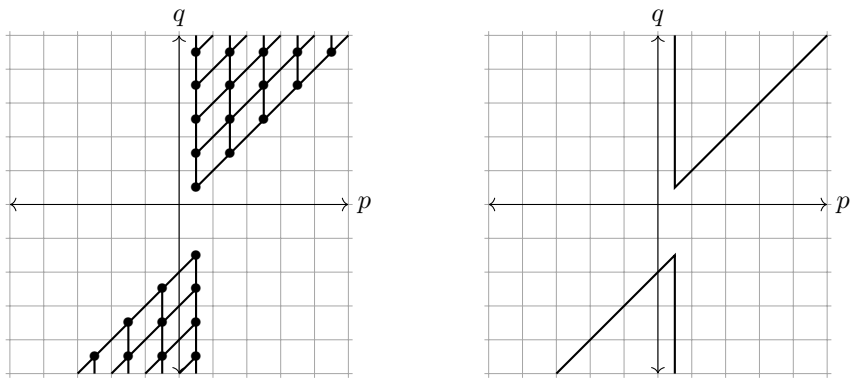
\begin{figure}[ht]
	\begin{tikzpicture}[scale=.45]
		\draw[help lines,light-gray] (-5.125,-5.125) grid (5.125, 5.125);
		\draw[<->] (-5,0)--(5,0)node[right]{$p$};
		\draw[<->] (0,-5)--(0,5)node[above]{$q$};
		\cone{0}{0}{black};
		\foreach \y in {0,1,2,3,4}
			\draw (0.5,\y+.5) node{\small{$\bullet$}};
		\foreach \y in {1,2,3,4}
			\draw (1.5,\y+.5) node{\small{$\bullet$}};
		\foreach \y in {2,3,4}
			\draw (2.5,\y+.5) node{\small{$\bullet$}};
		\foreach \y in {3,4}
			\draw (3.5,\y+.5) node{\small{$\bullet$}};
		\foreach \y in {4}
			\draw (4.5,\y+.5) node{\small{$\bullet$}};
		\foreach \y in {-1,-2,-3,-4}
			\draw (.5,\y-.5) node{\small{$\bullet$}};
		\foreach \y in {-2,-3,-4}
			\draw (-.5,\y-.5) node{\small{$\bullet$}};
		\foreach \y in {-3,-4}
			\draw (-1.5,\y-.5) node{\small{$\bullet$}};
		\foreach \y in {-4}
			\draw (-2.5,\y-.5) node{\small{$\bullet$}};
		\draw[thick](1.5,5)--(1.5,1.5);
		\draw[thick](2.5,2.5)--(2.5,5);
		\draw[thick](3.5,3.5)--(3.5,5);
		\draw[thick](4.5,4.5)--(4.5,5);
		\draw[thick](.5,1.5)--(4,5);
		\draw[thick](.5,2.5)--(3,5);
		\draw[thick](.5,3.5)--(2,5);
		\draw[thick](.5,4.5)--(1,5);
		\draw[thick](-.5,-2.5)--(-.5,-5);
		\draw[thick](-1.5,-3.5)--(-1.5,-5);
		\draw[thick](-2.5,-4.5)--(-2.5,-5);
		\draw[thick](.5,-2.5)--(-2,-5);
		\draw[thick](.5,-3.5)--(-1,-5);
		\draw[thick](.5,-4.5)--(0,-5);
	\end{tikzpicture}\hspace{0.5in}
	\begin{tikzpicture}[scale=.45]
		\draw[help lines,light-gray] (-5.125,-5.125) grid (5.125, 5.125);
		\draw[<->] (-5,0)--(5,0)node[right]{$p$};
		\draw[<->] (0,-5)--(0,5)node[above]{$q$};
		\cone{0}{0}{black};
	\end{tikzpicture}
	\caption{ The ring $\M_2=H^{*,*}(pt;\underline{\Z/2}$) with $H^{p,q}(pt;\underline{\Z/2})$ in spot $(p,q)$.}
\end{figure}

Let $S^n_a$ denote the $C_2$-space whose underlying space is $S^n$ and whose $C_2$-action is given by the antipodal map. Note $S^0_a$ is the free orbit $C_2$. We denote the cohomology of the space $S^n_a$ by $A_n$. This $\M_2$-module can be described algebraically as $A_n\cong \tau^{-1}\M_2/(\rho^{n+1})$. In the figure below, we illustrate the module structure of $A_0$, $A_1$, and $A_2$, respectively. As before, the detailed picture is shown on the left, while the abbreviated picture is shown to the right. The dots again indicate a copy of $\Z/2$ while the lines indicate the module structure. A full description of the module $A_n$ is given in Section \ref{ch:comptools}.

\begin{figure}[ht]
	\begin{tikzpicture}[scale=.45]
		\draw[help lines,light-gray] (-0.125,-5.125) grid (2.125, 5.125);
		\draw[<->] (-0,0)--(2,0)node[right]{$p$};
		\draw[<->] (0,-5)--(0,5)node[above]{$q$};
		\anti{0}{0}{black};
		\foreach \y in {-5,...,4}
			\draw (0.5,\y+.5) node{\small{$\bullet$}};
	\end{tikzpicture}
	\begin{tikzpicture}[scale=.45]
		\draw[help lines,light-gray] (-0.125,-5.125) grid (2.125, 5.125);
		\draw[<->] (-0,0)--(2,0)node[right]{$p$};
		\draw[<->] (0,-5)--(0,5)node[above]{$q$};
		\anti{0}{0}{black};
	\end{tikzpicture}\hspace{0.2 in}
	\begin{tikzpicture}[scale=.45]
		\draw[help lines,light-gray] (-0.125,-5.125) grid (2.125, 5.125);
		\draw[<->] (-0,0)--(2,0)node[right]{$p$};
		\draw[<->] (0,-5)--(0,5)node[above]{$q$};
		\anti{0}{1}{black};
		\foreach \y in {-5,...,4}
			{
			\draw (0.5,\y+.5) node{\small{$\bullet$}};
			\draw (1.5,\y+.5) node{\small{$\bullet$}};
			}
		\draw[thick] (.5,4.5)--(1,5);
		\draw[thick] (1,-5)--(1.5,-4.5);
	\end{tikzpicture}
	\begin{tikzpicture}[scale=.45]
		\draw[help lines,light-gray] (-0.125,-5.125) grid (2.125, 5.125);
		\draw[<->] (-0,0)--(2,0)node[right]{$p$};
		\draw[<->] (0,-5)--(0,5)node[above]{$q$};
		\anti{0}{1}{black};
		\draw[thick] (.5,4.5)--(1,5);
		\draw[thick] (1,-5)--(1.5,-4.5);
	\end{tikzpicture}\hspace{0.2in}
	\begin{tikzpicture}[scale=.45]
		\draw[help lines,light-gray] (-0.125,-5.125) grid (3.125, 5.125);
		\draw[<->] (-0,0)--(3,0)node[right]{$p$};
		\draw[<->] (0,-5)--(0,5)node[above]{$q$};
		\anti{0}{2}{black};
		\foreach \y in {-5,...,4}
			{
			\draw (0.5,\y+.5) node{\small{$\bullet$}};
			\draw (1.5,\y+.5) node{\small{$\bullet$}};
			\draw(2.5,\y+.5) node{\small{$\bullet$}};
			}
		\draw[thick] (.5,4.5)--(1,5);
		\draw[thick] (1,-5)--(2.5,-3.5);
		\draw[thick] (1.5,5)--(1.5,-5);
		\draw[thick] (.5,3.5)--(2,5);
		\draw[thick] (2,-5)--(2.5,-4.5);
	\end{tikzpicture}
	\begin{tikzpicture}[scale=.45]
		\draw[help lines,light-gray] (-0.125,-5.125) grid (3.125, 5.125);
		\draw[<->] (-0,0)--(3,0)node[right]{$p$};
		\draw[<->] (0,-5)--(0,5)node[above]{$q$};
		\anti{0}{2}{black};
		\draw[thick] (.5,4.5)--(1,5);
		\draw[thick] (1,-5)--(2.5,-3.5);
		\draw[thick] (.5,3.5)--(2,5);
		\draw[thick] (2,-5)--(2.5,-4.5);
	\end{tikzpicture}	
	
	\caption{The $\M_2$-modules $A_0$, $A_1$, and $A_2$, repsectively.}
\end{figure}
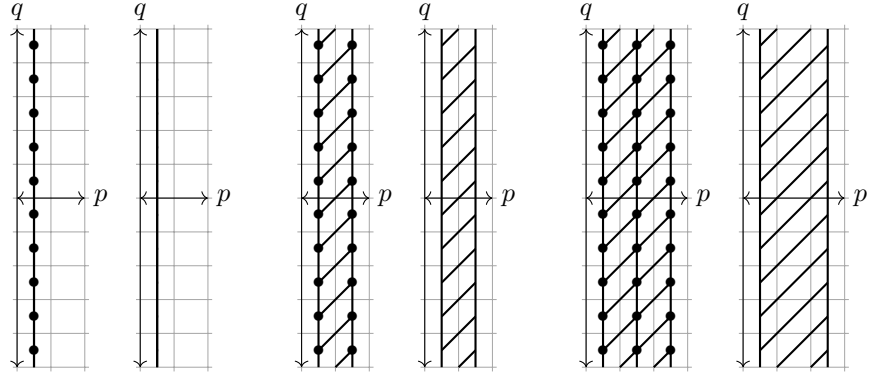

It is shown in \cite{M1} that as an $\M_2$-module, the cohomology any finite $C_2$-CW complex is isomorphic to a direct sum of shifted copies of $\M_2$ and shifted copies of $A_j$ for some values of $j$. Our goal in this paper is to find the specific decompositions for the cohomology of all $C_2$-surfaces. We begin by showing a few examples.

\subsection{Nonfree Examples} 
To give the reader a flavor of the sorts of decompositions that can appear, we provide three examples of $C_2$-surfaces and state each of their cohomologies.

\subsubsection*{Example 1} Let $X_1$ denote the $C_2$-space whose underlying space is the genus one torus, and whose action is given by the reflection action. The space $X_1$ is depicted below with the fixed set $X_1^{C_2}$ shown in {\textbf{\color{blue}{blue}}}. We will eventually show as an $\M_2$-module
	\[
	H^{*,*}(X_1;\underline{\Z/2}) \cong \M_2\oplus \Sigma^{1,0}\M_2 \oplus \Sigma^{1,1} \M_2 \oplus \Sigma^{2,1}\M_2.
	\]
This module is illustrated below. 
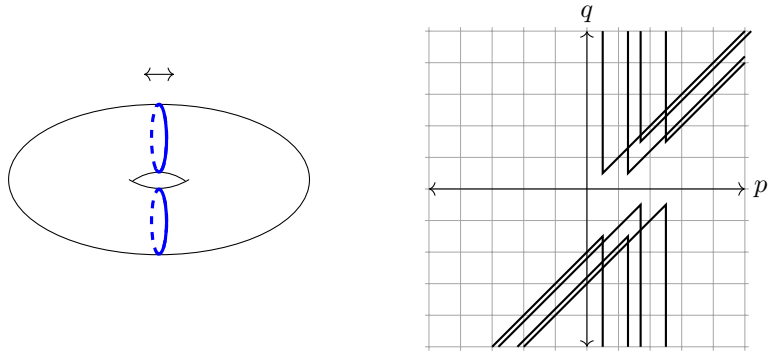
\begin{figure}[ht]
	\begin{subfigure}[b]{0.45\textwidth}
	\centering
	\begin{tikzpicture}[scale=1]
		\draw[<->] (-.2,1.4)--(.2,1.4);
		\draw (0,0) ellipse (2cm and 1cm);
		\draw (-0.4,0) to[out=330,in=210] (.4,0) ;
		\draw (-.35,-.02) to[out=35,in=145] (.35,-.02);
		\draw[very thick,blue] (0,.55) ellipse (.1cm and 0.45cm);
		\draw[white,fill=white] (-.3,0.15) rectangle (0,.96);
		\draw[very thick,blue,dashed] (0,.55) ellipse (.1cm and 0.45cm);
		\draw[very thick,blue] (0,-.56) ellipse (.1cm and 0.43cm);
		\draw[white,fill=white] (-.3,-0.95) rectangle (0,-.14);
		\draw[very thick,blue,dashed] (0,-.56) ellipse (.1cm and 0.43cm);
	\end{tikzpicture}
	\vspace{0.5in}
	\end{subfigure}
	\begin{subfigure}[b]{0.45\textwidth}
	\centering
	\begin{tikzpicture}[scale=0.42]
		\draw[help lines,light-gray] (-5.125,-5.125) grid (5.125, 5.125);
		\draw[<->] (-5,0)--(5,0)node[right]{$p$};
		\draw[<->] (0,-5)--(0,5)node[above]{$q$};
		\cone{0}{0}{black};
		\cone{1.2}{1}{black};
		\draw[thick] (1.3,5)--(1.3,.5)--(5,4.2);
		\draw[thick] (1.3,-5)--(1.3, -1.5)--(-2.2,-5);
		\draw[thick] (2.5,5)--(2.5,1.5)--(5,4);
		\draw[thick] (2.5,-5)--(2.5,-.5)--(-2,-5);
	\end{tikzpicture}
	\end{subfigure}
	\caption{The space $X_1$ and its cohomology.}
\end{figure}

Observe the cohomology is free over $\M_2$, and there is exactly one generator in topological dimension zero, exactly two generators in topological dimension one, and exactly one generator in topological dimension two (recall $p$ is the topological dimension). This should be unsurprising based on the singular cohomology of the torus, though the weights of these generators are more mysterious.

\subsubsection*{Example 2} Let $X_2$ denote the $C_2$-surface whose underlying space is the genus $7$ torus, and whose $C_2$-action is given by the rotation action depicted below. The fixed set consists of $8$ isolated points that are shown in blue. The cohomology of $X_2$ is given by
	\[
	H^{*,*}(X_2;\underline{\Z/2}) \cong \M_2 \oplus \left(\Sigma^{1,1}\M_2\right)^{\oplus 6} \oplus \left(\Sigma^{1,0} A_0\right)^{\oplus 4} \oplus \Sigma^{2,2}\M_2.
	\]
The module is illustrated below. 
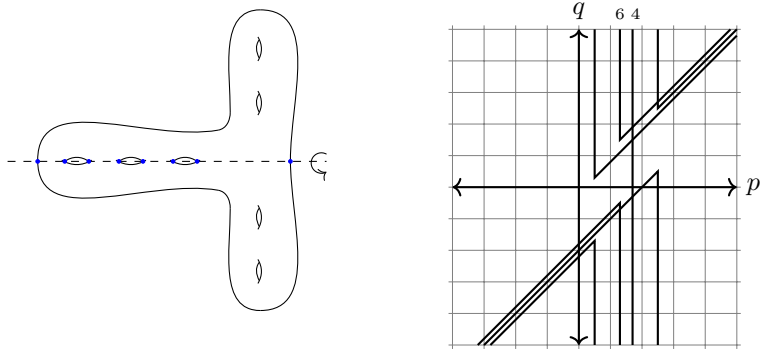
\begin{figure}[ht]
\begin{subfigure}[b]{0.45\textwidth}
	\centering
	\begin{tikzpicture}[scale=0.8]
		\draw[dashed](-2,-0.52)--(3.3,-0.52);
		\draw[->] (3.3,-0.4) arc (60:300:.17cm);
		\draw (1.5,0) to[out=190,in=90] (-1.5,-.5);
		\draw (1.5,-1) to[out=170,in=270] (-1.5,-.5);
		\draw (1.5,0) to[out=10,in=270] (1.7,0.3);
		\draw (1.7,0.3) to[out=90,in=180] (2.2,2);
		\draw (2.2,2) to[out=0,in=90] (2.7,-0.5);
		\draw (2.7,-0.5) to[out=270,in=0] (2.2,-3);
		\draw (2.2,-3) to [out=180,in=270] (1.7,-1.3);
		\draw (1.7,-1.3) to [out=90,in=350] (1.5,-1);
		\draw (-1.1,-.5) to[out=330,in=210] (-.6,-.5) ;
		\draw (-1.05,-.52) to[out=35,in=145] (-.65,-.52);
		\draw (-0.2,-.5) to[out=330,in=210] (.3,-.5) ;
		\draw (-.15,-.52) to[out=35,in=145] (.25,-.52);
		\draw (0.7,-.5) to[out=330,in=210] (1.2,-.5) ;
		\draw (.75,-.52) to[out=35,in=145] (1.15,-.52);
		\draw (2.16,1.15) to[out=60,in=300] (2.16,1.55);
		\draw (2.18,1.2) to[out=115,in=235] (2.18,1.5);
		\draw (2.16,0.25) to[out=60,in=300] (2.16,0.65);
		\draw (2.18,0.3) to[out=115,in=235] (2.18,0.6);
		\draw (2.16,-1.65) to[out=60,in=300] (2.16,-1.25);
		\draw (2.18,-1.6) to[out=115,in=235] (2.18,-1.3);
		\draw (2.16,-2.55) to[out=60,in=300] (2.16,-2.15);
		\draw (2.18,-2.5) to[out=115,in=235] (2.18,-2.2);
		\draw[blue,fill=blue] (-1.5,-.52) circle (0.03cm);
		\draw[blue,fill=blue] (-1.05,-.52) circle (0.03cm);
		\draw[blue,fill=blue] (-.65,-.52) circle (0.03cm);
		\draw[blue,fill=blue] (-.15,-.52) circle (0.03cm);
		\draw[blue,fill=blue] (0.25,-.52) circle (0.03cm);
		\draw[blue,fill=blue] (0.75,-.52) circle (0.03cm);
		\draw[blue,fill=blue] (1.15,-.52) circle (0.03cm);
		\draw[blue,fill=blue] (2.7,-.52) circle (0.03cm);
	\end{tikzpicture}
	\vspace{0.2in}
\end{subfigure}
\begin{subfigure}[b]{0.45\textwidth}
	\centering
	\begin{tikzpicture}[scale=0.42]
		\draw[help lines,gray] (-4.125,-5.125) grid (5.125, 5.125);
		\draw[thick,<->] (-4,0)--(5,0)node[right]{$p$};
		\draw[thick,<->] (0,-5)--(0,5)node[above]{$q$};
		\draw[thick] (0.5,5)--(0.5,0.3)--(5,4.8);
		\draw[thick] (0.5,-5)--(0.5,-1.7)--(-2.8,-5);
		\draw[thick] (1.3,5)--(1.3,1.5)--(4.8,5);
		\draw[thick] (1.3,-5)--(1.3,-.5)--(-3.2,-5);
		\lab{.8}{6}{black};
		\draw[thick] (1.7,5)--(1.7,-5);
		\lab{1.3}{4}{black};
		\draw[thick] (2.5,-5)--(2.5,0.5)--(-3,-5);
		\draw[thick] (2.5,5)--(2.5,2.5)--(5,5);
	\end{tikzpicture}
\end{subfigure}
	\caption{The space $X_2$ and its cohomology.}
\end{figure}

Again there is exactly one free generator in topological dimensions zero and two, but there is something more interesting going on in topological dimension one. There are four nonfree summands, and six free summands in weight one. 

\subsubsection*{Example 3} 
Our last example $X_3$ is the $C_2$-space whose underlying space is $\R P^2$ and whose $C_2$-action is depicted below. Note the fixed set contains both a fixed circle and a fixed point; this did not happen in the previous examples. The cohomology of this space is given by
	\[
	H^{*,*}(X_3;\underline{\Z/2}) \cong \M_2 \oplus \Sigma^{1,1}\M_2 \oplus \Sigma^{2,1}\M_2.
	\]
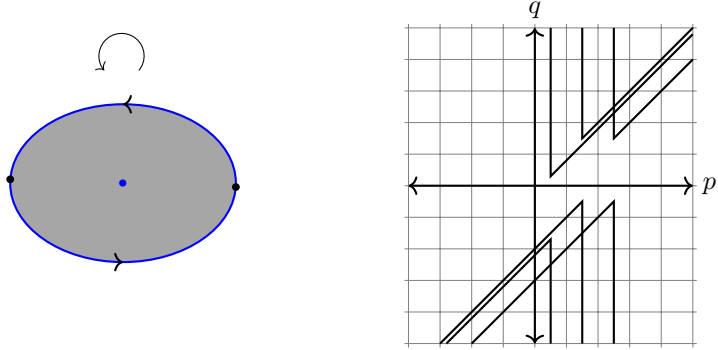
\begin{figure}[ht]
	\begin{subfigure}[b]{0.45\textwidth}
	\centering
	\begin{tikzpicture}[scale=1.5]
		\draw[blue,thick, fill=light-gray, decoration={markings, mark=at position 0 with {\arrow[black,scale=.5]{*}}, mark=at position 0.25 with {\arrow[black]{>}}, mark=at position 0.5 with {\arrow[black,scale=.5]{*}}, mark=at position 0.75 with {\arrow[black]{>}}}, postaction={decorate}] (0,0) ellipse (1 cm and .7 cm);
		\draw[blue] (0,0)node {\tiny{$\bullet$}};
		\draw[->] (.14,1) arc (-40:220:.20cm);
	\end{tikzpicture}
	\vspace{0.4in}
	\end{subfigure}
	\begin{subfigure}[b]{0.45\textwidth}
	\centering
	\begin{tikzpicture}[scale=0.42]
		\draw[help lines,gray] (-4.125,-5.125) grid (5.125, 5.125);
		\draw[thick,<->] (-4,0)--(5,0)node[right]{$p$};
		\draw[thick,<->] (0,-5)--(0,5)node[above]{$q$};
		\draw[thick] (0.5,5)--(0.5,0.3)--(5,4.8);
		\draw[thick] (0.5,-5)--(0.5,-1.7)--(-2.8,-5);
		\draw[thick] (1.5,5)--(1.5,1.5)--(5,5);
		\draw[thick](1.5,-5)--(1.5,-.5)--(-3,-5);
		\draw[thick] (2.5,-5)--(2.5,-.5)--(-2,-5);
		\draw[thick] (2.5,5)--(2.5,1.5)--(5,4);
	\end{tikzpicture}
	\end{subfigure}
	\caption{The space $X_3$ and its cohomology.}
\end{figure}

\smallskip
With some care, one can compute the cohomology of the above spaces by hand using tools given in Section \ref{ch:comptools}. Though, it is not obvious how we could have predicted these answers by just looking at the spaces. What properties are being detected by the cohomology? In order to answer this question, we next define some invariants of $C_2$-surfaces, some of which will be detected by the cohomology.

\subsection{Invariants of $C_2$-surfaces} 
There are a handful of invariants that can be associated to a given $C_2$-surface. For example, we can count the number of isolated fixed points, or the number of fixed circles. Recall for a nonequivariant surface $X$ the homeomorphism type is entirely determined by two invariants, namely $\dim_{\Z/2}H^{1}_{sing}(X;\Z/2)$ and whether or not $X$ is orientable. In what follows, we will denote $\dim_{\Z/2}H^{1}_{sing}(X;\Z/2)$ by $\beta(X)$ and refer to this as the \emph{$\beta$-genus} of $X$.

In \cite{D2} it is shown there is a list of invariants that uniquely determines the isomorphism type of a $C_2$-action on a given surface. We recall two of these invariants below.
\begin{definition} Let $X$ be a $C_2$-surface. We can associate the following invariants to $X$:
	\begin{enumerate}
		\item[(i)] $F(X)$ is the number of isolated fixed points;
		\item[(ii)] $C(X)$ is the number of fixed circles.
	\end{enumerate}
\end{definition}
When there is no ambiguity about the space we are discussing, we will simply write $F$ instead of $F(X)$ and $C$ instead of $C(X)$. There are four other invariants needed in the classification given in \cite{D2} that we will not define here, but do note the invariants defined above are not enough to uniquely determine the isomorphism type in general. 

In the examples above, we can compute
	\[
	\beta(X_1)=2,\quad F(X_1)=0, \quad C(X_1)=2;
	\]
	\[
	\beta(X_2)=14,\quad F(X_2)=8, \quad C(X_2)=0;
	\]
	\[
	\beta(X_3)=1,\quad F(X_3)=1,\quad C(X_3)=1.
	\]
This paper addresses the following questions: how does the cohomology of a $C_2$-surface relate to the invariants? Better yet, can we find a formula that gives the cohomology of $X$ based on some of these invariants? It may not be apparent how to do this based on the three given examples, but the answer to the latter is yes, as explained below.

\subsection{The Answer for Nonfree $C_2$-surfaces}
We now state the decompositions for nontrivial, nonfree $C_2$-surfaces in \underline{$\Z/2$}-coefficients. Recall the fixed set of an involution on a surface is always given by a disjoint union of isolated points and copies of $S^1$. Recall $\M_2$ denotes the cohomology of a point and $A_0$ denotes the cohomology of the free orbit $C_2$. We will show the following:
\newpage
\begin{theorem}\label{intro1} Let $X$ be a nontrivial, nonfree $C_2$-surface. There are two cases for the $RO(C_2)$-graded Bredon cohomology of $X$ in \underline{$\Z/2$}-coefficients.
\begin{enumerate}
	\item[(i)] Suppose $C=0$. Then
		\[
		H^{\ast, \ast}(X; \underline{\Z/2})\cong \M_2 \oplus \left(\Sigma^{1,1} \M_2\right)^{\oplus F-2} \oplus \left(\Sigma^{1,0}A_0 \right)^{\oplus \frac{\beta-F}{2}+1} \oplus \Sigma^{2,2}\M_2
		\]
	\item[(ii)] Suppose $C\neq 0$. Then 
		\begin{align*}
			H^{\ast, \ast}(X; \underline{\Z/2})\cong & ~\M_2 \oplus \left(\Sigma^{1,1} \M_2\right)^{\oplus F+C-1} \oplus  \left(\Sigma^{1,0} \M_2\right)^{\oplus C-1}\\
			&\oplus \left(\Sigma^{1,0}A_0 \right)^{\oplus \frac{\beta-F}{2}+1-C}\oplus \Sigma^{2,1}\M_2
		\end{align*}
\end{enumerate}
\end{theorem}

We invite the reader to check that the above formulas match the answers given in the three examples. Note not all of the invariants given in \cite{D2} are needed to determine the cohomology. In particular, the module structure of the cohomology in $\underline{\Z/2}$-coefficients does not determine the isomorphism type of a $C_2$-surface. 

\begin{rmk} As observed in our examples, there is exactly one summand generated in topological dimension zero, exactly one summand generated in topological dimension two, and some number of summands appearing in topological dimension one. Based on what we know of the singular cohomology of surfaces, this shouldn't be surprising. Though, the exact number and type of summands generated in topological dimension one is nonobvious. Note we can recover dimension of the singular cohomology of $X$ by
	\[
	\beta=2\cdot \# \left(\Sigma^{1,0}A_0 \text{-summands}\right) + \# \left(\Sigma^{1,0}\M_2 \text{-summands}\right) + \# \left(\Sigma^{1,1}\M_2\text{-summands}\right).
	\]
\end{rmk}

\subsection{Equivariant Connected Sum}
Before stating the decomposition for free actions, we need to define one construction. Given a nontrivial equivariant surface $X$ and a nonequivariant surface $Y$, we can form the equivariant connected sum $X\#_2 Y$ as follows. Let $Y'$ denote the space obtained by removing a small disk from $Y$. Let $D$ be a disk in $X$ that is disjoint from its conjugate disk, $\sigma D$, and let $X'$ denote the space obtained by removing both of these disks. Choose an isomorphism $f:\partial Y' \to \partial D$. Then the space $X \#_2 Y$ is given by 
\[[\left(Y' \times \{0\}\right) \sqcup \left(Y' \times \{1\} \right)\sqcup X']/\sim\]
where $(y,0) \sim f(y)$ and $(y,1) \sim \sigma(f(y))$ for $y\in \partial Y'$. Note nonequivariantly, $X\#_2Y\cong Y\#X\#Y$.
Below is an example of when $X=S^2_a$ and $Y=T_1$ is the genus one torus.
\begin{figure}[ht]
	\begin{tikzpicture}[scale=.8]
		\draw (0,0) to[out=40,in=140] (2,0);
		\draw[thick, red] (2,0) to[out=320,in=210] (2.5,0);
		\draw[thick, red](2.5,0) to[out=20,in=90] (3.5,-.5);
		\draw (0,-1) to[out=320,in=220] (2,-1);
		\draw[thick, red] (2,-1) to[out=40,in=160] (2.5,-1);
		\draw[thick, red] (2.5,-1) to[out=340,in=270] (3.5,-.5);
		\draw[thick, red] (0,0) to[out=220,in=330] (-.5,0);
		\draw[thick, red] (-.5,0) to[out=160,in=90] (-1.5,-.5);
		\draw[thick, red] (0,-1) to[out=140,in=20] (-.5,-1);
		\draw[thick, red] (-.5,-1) to[out=200,in=270] (-1.5,-.5);
		\draw[thick, red] (2.6,-.5) to[out=330,in=210] (3.1,-.5);
		\draw[thick, red] (2.65,-.52) to[out=35,in=145] (3.05,-.52);
		\draw[thick, red] (-1.1,-.5) to[out=330,in=210] (-.6,-.5) ;
		\draw[thick, red] (-1.05,-.52) to[out=35,in=145] (-.65,-.52);
		\draw[thick, red] (0,0) to[out=240,in=120] (0,-1);
		\draw[thick, dashed,red] (0,0) to[out=300,in=60] (0,-1);
		\draw[thick, dashed,red] (2,0) to[out=240,in=120] (2,-1);
		\draw[thick, red] (2,0) to[out=300,in=60] (2,-1);
		\draw (1,.38) to[out=240,in=120] (1,-1.38);
		\draw[dashed] (1,.38) to[out=300,in=60] (1,-1.38);
	\end{tikzpicture}
	\caption{$S^2_a\#_2 {\mathbf{\color{red}{T_1}}}$}
\end{figure}

In \cite{D2}, it is shown that there is exactly one free action on the sphere up to equivariant isomorphism, namely the antipodal action, and there are exactly two free actions on the torus, namely the antipodal action and the action given by rotating 180$^\circ$ around an axis through the center of the hole of the torus. Interestingly, it is also shown for every free $C_2$-surface $X$ there is a surface $Y$ such that $X$ is isomorphic to $Z\#_2 Y$ where $Z$ is either the free $C_2$-sphere or one of the two free $C_2$-tori. 

\subsection{The Answer for Free $C_2$-surfaces}
We can now state the theorem for free $C_2$-surfaces. Recall $A_n$ denotes the cohomology of $S^n$ with the antipodal action. We will prove the following: 
\begin{theorem}\label{intro2} Let $X$ be a free $C_2$-surface. There are two cases for the $RO(C_2)$-graded Bredon cohomology of $X$ in \underline{$\Z/2$}-coefficients.
\begin{enumerate}
	\item[(i)] Suppose $X$ is equivariantly isomorphic to $S^2_a\#_2Y$ where $Z$ is the free $C_2$-sphere. Then
	\[
	H^{\ast, \ast}(X; \underline{\Z/2})\cong  \left( \Sigma^{1,0}A_0 \right)^{\oplus \beta(X)/2}\oplus A_2 
	\]
	\item[(ii)] Suppose $X$ is equivariantly isomorphic to $Z\#_2 Y$ where $Z$ is a free $C_2$-torus. Then
	\[
	H^{\ast, \ast}(X; \underline{\Z/2})\cong  \left( \Sigma^{1,0}A_0 \right)^{\oplus \frac{\beta(X)-2}{2}}\oplus A_1 \oplus \Sigma^{1,0}A_1 
	\]
\end{enumerate}
\end{theorem}\

\noindent For example, the space $S^2_a\#_2 T_1$ shown above would have cohomology given by
	\[
	H^{*,*}(S^2_a\#_2 T_1) \cong \left(\Sigma^{1,0}A_0\right)^{\oplus 2} \oplus A_2.
	\]

\begin{rmk} It is not clear a priori why the number of summands of $\Sigma^{1,0}A_0$ given in Theorem \ref{intro1} and Theorem \ref{intro2} is necessarily an integer. This nontrivial fact follows from restrictions on $\beta$, $F$, and $C$ that arise in the classification of $C_2$-surfaces given in \cite{D2}.   
\end{rmk}

\begin{rmk} There are two natural questions given the above theorems. First, one may wonder if there is a geometric interpretation of these results, and specifically, if certain cohomology classes can be represented by equivariant cycles. Indeed, the author has developed a theory of fundamental classes for equivariant submanifolds of $C_2$-manifolds in $\underline{\Z/2}$-coefficients, and it has been shown the cohomology of any $C_2$-surface is generated by such classes. These classes also allow us to easily compute the ring structure. These results can be found in \cite{H1}. It is also natural to ask about the cohomology of $C_2$-surfaces with coefficients in other Mackey functors, such as constant integer coefficients. The orientable case in $\underline{\Z}$-coefficients was done in \cite{LFdS2}. This computation was redone using methods similar to those in this paper and also expanded to all $C_2$-surfaces. These integral computations will appear in \cite{H2}. 
\end{rmk}

\subsection{Organization of the Paper} 
In Section \ref{ch:preliminaries} we give some preliminary facts about Bredon cohomology. Important computational tools used throughout the paper are introduced in Section \ref{ch:comptools}. We then compute the cohomology of all free $C_2$-surfaces and prove Theorem \ref{intro2} in Section \ref{ch:freecomp}. The computation for nonfree surfaces and the proof of Theorem \ref{intro1} are in Section \ref{ch:nonfreecomp}. Lastly Appendix \ref{ch:manfacts} has a proof of a general statement about the cohomology of $C_2$-manifolds that is mentioned in Section \ref{ch:preliminaries}.

\subsection{Acknowledgements}
The work in this paper is part of the author's thesis project at University of Oregon. The author would like to thank her doctoral advisor Dan Dugger for all of his guidance and for many helpful conversations. 
%
%
\section{Preliminaries}\label{ch:preliminaries}
In this section, we review some preliminary facts about $RO(G)$-graded Bredon cohomology in the case of $G=C_2$. Our coefficients are given by Mackey functors, so we provide a definition of a Mackey functor in the case of $G=C_2$ and review the Mackey functor that will be used throughout the paper. We next review how the cohomology theory is a bigraded theory, and how this lends itself to pictorially representing various module and ring structures. 

\subsection{Mackey functors} 
The coefficients of $RO(G)$-graded Bredon cohomology are what is known as a Mackey functor. In general, the definition of a Mackey functor requires some work, and a general exposition of Mackey functors can be found in \cite{S} or in \cite{M2}. In the case of $G=C_2$, the definition can be distilled to the following.

\begin{definition} 
A \textbf{Mackey functor} $M$ for $G=C_2$ is the data of 
\begin{center}
\begin{tikzcd}[column sep=tiny]
M:&M(C_2)\arrow[out=-120, in=-60, loop, distance=0.5cm, "t^*" below ] \arrow[rr,shift right=.6ex,"p_*" below] & ~&M(\ast)\arrow[ll,shift right=.6ex,"p^*" above]
\end{tikzcd}
\end{center}
where $M(C_2)$ and $M(\ast)$ are abelian groups, and $p^*$, $p_*$, $t^*$ are homomorphisms that satisfy 
\begin{enumerate}
	\item[(i)] $(t^*)^2 = id$,
	\item[(ii)] $t^*\circ p^*=p^*$,
	\item[(iii)] $p_*\circ t^*=p_*$, and
	\item[(iv)] $p^*\circ p_*=1+t^*$.
\end{enumerate}
\end{definition}

Given an abelian group $B$, we can form the \emph{constant Mackey functor} $\underline{B}$ where $\underline{B}(C_2)=\underline{B}(\ast)=B$, $t^*=id$, $p_*=2$, and $p^*=id$. We will be concerned with the following constant Mackey functor.
\begin{center}
	\begin{tikzcd}[column sep=tiny]
	\underline{\Z/2}:& \Z/2\arrow[out=-120, in=-60, loop, distance=0.5cm, "1" below ] 	\arrow[rr,shift right=.6ex,"0" below] & ~& \Z/2\arrow[ll,shift right=.6ex,"1" above]
	\end{tikzcd}
\end{center}
\medskip

\subsection{Bigraded Theory}\label{2.3} 
For a group $G$, Bredon cohomology is graded on $RO(G)$, the Grothendieck ring of finite-dimensional, real, orthogonal $G$-representations. When $G$ is the cyclic group of order two, observe any such $C_2$-representation $V$ is isomorphic to a direct sum of copies of the trivial representation $\R_{triv}$ and copies of the sign representation $\R_{sgn}$. Up to isomorphism, $V$ is entirely determined by its dimension and the number of sign representations appearing in this decomposition. It follows that $RO(C_2)$ is a rank $2$ free abelian group with generators given by $[\R_{triv}]$ and $[\R_{sgn}]$. For brevity, we will write $\R^{p,q}$ for the $p$-dimensional representation $\R_{triv}^{p-q}\oplus \R_{sgn}^{q}$. We will also write $\R^{p,q}$ for the element of $RO(C_2)$ that is equal to $(p-q)[\R_{triv}] + q[\R_{sqn}]$. When computing cohomology groups, we will write $H^{p,q}(X;M)$ for the cohomology group $H^{\R^{p,q}}(X;M)$. Note some authors have different grading conventions for $RO(C_2)$, and here we are using what is known as the motivic grading. 

Given any finite-dimensional, real, orthogonal, $G$-representation $V$ we can form the one-point compactification $\hat{V}$. Note this new space will be an equivariant sphere which we will denote $S^V$; such spaces are referred to as \emph{representation spheres}. Using these representation spheres, we can form equivariant suspensions. Whenever we have a based $G$-space $X$, we can form the $V$-th suspension of $X$ by
	\[
	\Sigma^VX=S^V \wedge X.
	\]
Note the basepoint must be a fixed point. Often when working with free spaces we will add a disjoint basepoint in order to form suspensions and cofiber sequences. We use the common notation of $X_+$ for $X \sqcup \{*\}$ where the disjoint basepoint is understood to be fixed by the action.

An important feature of Bredon cohomology is that we have suspension isomorphisms: given any finite dimensional, real, orthogonal $G$-representation, there are natural isomorphisms
	\[
	\Sigma^{V}: \tilde{H}^{\alpha}(-;M) \to \tilde{H}^{\alpha+V}(\Sigma^V(-);M).
	\]
Given a cofiber sequence of based $G$-spaces
	\[
	A\overset{f}{\to} X \to C(f)
	\] 
we can form the Puppe sequence
	\[
	A\to X \to C(f) \to \Sigma^{\mathbf{1}}A \to \Sigma^{\mathbf{1}}C(f) \to \Sigma^{\mathbf{1}}X \to \dots 
	\]
where $\mathbf{1}$ is the one-dimensional trivial representation. From the suspension isomorphism this yields a long exact sequence
	\[
	\H^{V}(A)\leftarrow \H^{V}(X)\leftarrow \H^{V}(C(f)) \leftarrow \H^{V-\mathbf{1}}(A)\leftarrow \H^{V-\mathbf{1}}(X)\leftarrow \dots
	\]
for each representation $V\in RO(G)$. We will make use of such long exact sequences throughout the paper.

When $G=C_2$, we have already discussed how the Bredon cohomology theory is a bigraded theory, and we will carry this notation over when discussing representation spheres and equivariant suspensions. In particular, we will denote $S^{\R^{p,q}}$ by $S^{p,q}$ and for a based space $X$ we will denote $\Sigma^{\R^{p,q}}X$ by $\Sigma^{p,q}X$. Translating the above into this notation, we have natural isomorphisms
\[\Sigma^{p,q}:\tilde{H}^{a,b}(-;M) \to \tilde{H}^{a+p,b+q}(\Sigma^{p,q}(-);M)\]
for all $p,q\geq 0$. Given a cofiber sequence we have long exact sequences
\[\dots\to\H^{p,q}(C(f))\to H^{p,q}(X)\to H^{p,q}(A) \to \H^{p+1,q}(C(f))\to H^{p+1,q}(X)\to \dots\]
for each $q\in \Z$.

During our computations, three particular representation spheres will appear often, namely $S^{1,1}$, $S^{2,1}$, and $S^{2,2}$. We include an illustration of these equivariant spheres below in Figure \ref{fig:spheres}. The fixed set is shown in {\textbf{\color{blue}{blue}}} while the arrow is used to indicate the action of $C_2$ on the space. \smallskip

\begin{figure}[ht]
	\begin{tikzpicture}
		\draw (0,0) circle (1cm);
		\draw[blue] (0,-1) node{$\bullet$};
		\draw[blue] (0,1) node{$\bullet$};
		\draw[dashed] (0,1.4)--(0,-1.4);
		\draw[<->] (-.25,1.2)--(.25,1.2);
		\draw (0,-1.7) node{$S^{1,1}$};
	\end{tikzpicture}\hspace{.5in}
	\begin{tikzpicture}
		\draw (0,0) circle (1cm);
		\draw[blue,very thick] (-1,0) to[out=330,in=210](1,0);
		\draw[blue,very thick, dashed] (-1,0) to[out=30,in=150](1,0);
		\draw (0,-1.7) node{$S^{2,1}$};
		\draw[<->] (1.4,-.2)--(1.4,.2);
	\end{tikzpicture}\hspace{.2in}
	\begin{tikzpicture}
		\draw (0,0) circle (1cm);
		\draw (-1,0) to[out=330,in=210](1,0);
		\draw[dashed] (-1,0) to[out=30,in=150](1,0);
		\draw[dashed] (-1.4,0)--(1.4,0);
		\draw[->] (1.5,.15) arc (60:300:.17cm);
		\draw (0,-1.7) node{$S^{2,2}$};
		\draw[blue] (1,0)node{$\bullet$};
		\draw[blue] (-1,0)node{$\bullet$};
	\end{tikzpicture}
	\caption{Some representation spheres.}
	\label{fig:spheres}
\end{figure}
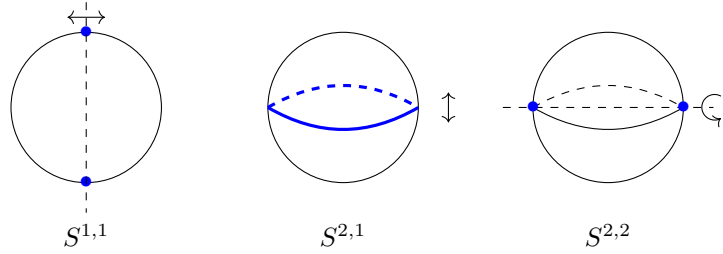

\subsection{The Cohomology of Orbits} 
Given any $C_2$-space $X$ we have an equivariant map $X\to pt$ where $pt$ denotes a single point with the trivial action. On cohomology, this gives a map of rings $H^{*,*}(pt;\underline{\Z/2}) \to H^{*,*}(X;\underline{\Z/2})$. Thus the cohomology of $X$ is a module over the cohomology of a point, which recall we denote $\M_2=H^{*,*}(pt;\underline{\Z/2})$. In this paper, we will be computing the cohomology of various spaces as $\M_2$-modules. Below we describe the cohomology of $pt=C_2/C_2$ as well as the cohomology of the free orbit $C_2$. These computations have been done many times and are often attributed to unpublished notes of Stong.  The computation for coefficients in any constant Mackey functor can be found in \cite{L}. A computation for constant integer coefficients can also be found in Appendix B of \cite{D1}, and the same methods used there can be used to compute the cohomology of orbits in constant $\Z/2$ coefficients. 

 In $\underline{\Z/2}$-coefficients, the cohomology of a point is illustrated in the left-hand grid shown in Figure \ref{fig:point}. The $(p,q)$ spot on the grid refers to the $\R^{p,q}$-cohomology group. Each dot represents a copy of $\Z/2$, and we adopt the convention that the $(p,q)$ group is plotted up and to right of the $(p,q)$ coordinate. For example, $H^{0,0}(pt;\underline{\Z/2})$ is isomorphic to $\Z/2$, while $H^{1,0}(pt;\underline{\Z/2})$ is zero. 

\begin{figure}[ht]
	\begin{tikzpicture}[scale=.45]
		\draw[help lines,light-gray] (-5.125,-5.125) grid (5.125, 5.125);
		\draw[<->] (-5,0)--(5,0)node[right]{$p$};
		\draw[<->] (0,-5)--(0,5)node[above]{$q$};
		\cone{0}{0}{black};
		\foreach \y in {0,1,2,3,4}
			\draw (0.5,\y+.5) node{\small{$\bullet$}};
		\foreach \y in {1,2,3,4}
			\draw (1.5,\y+.5) node{\small{$\bullet$}};
		\foreach \y in {2,3,4}
			\draw (2.5,\y+.5) node{\small{$\bullet$}};
		\foreach \y in {3,4}
			\draw (3.5,\y+.5) node{\small{$\bullet$}};
		\foreach \y in {4}
			\draw (4.5,\y+.5) node{\small{$\bullet$}};
		\foreach \y in {-1,-2,-3,-4}
			\draw (.5,\y-.5) node{\small{$\bullet$}};
		\foreach \y in {-2,-3,-4}
			\draw (-.5,\y-.5) node{\small{$\bullet$}};
		\foreach \y in {-3,-4}
			\draw (-1.5,\y-.5) node{\small{$\bullet$}};
		\foreach \y in {-4}
			\draw (-2.5,\y-.5) node{\small{$\bullet$}};
		\draw[thick](1.5,5)--(1.5,1.5)node[below, right]{$\rho$};
		\draw[thick](2.5,2.5)--(2.5,5);
		\draw[thick](3.5,3.5)--(3.5,5);
		\draw[thick](4.5,4.5)--(4.5,5);
		\draw[thick](.5,1.5)node[xshift=-2.2ex]{$\tau$}--(4,5);
		\draw[thick](.5,2.5)--(3,5);
		\draw[thick](.5,3.5)--(2,5);
		\draw[thick](.5,4.5)--(1,5);
		\draw[thick](-.5,-2.5)--(-.5,-5);
		\draw[thick](-1.5,-3.5)--(-1.5,-5);
		\draw[thick](-2.5,-4.5)--(-2.5,-5);
		\draw[thick](.5,-2.5)--(-2,-5);
		\draw[thick](.5,-3.5)--(-1,-5);
		\draw[thick](.5,-4.5)--(0,-5);
		\draw (1,-1.5) node{$\theta$};
		\draw (1,-2.5) node{$\frac{\theta}{\tau}$};
		\draw (1,-3.5) node{$\frac{\theta}{\tau^2}$};
		\draw (1,-4.5) node{$\frac{\theta}{\tau^3}$};
		\draw (-.7,-1.7) node{$\frac{\theta}{\rho}$};
		\draw (-1.7,-2.7) node{$\frac{\theta}{\rho^2}$};
		\draw (-2.7,-3.7) node{$\frac{\theta}{\rho^3}$};
	\end{tikzpicture}\hspace{0.5in}
	\begin{tikzpicture}[scale=.45]
		\draw[help lines,light-gray] (-5.125,-5.125) grid (5.125, 5.125);
		\draw[<->] (-5,0)--(5,0)node[right]{$p$};
		\draw[<->] (0,-5)--(0,5)node[above]{$q$};
		\cone{0}{0}{black};
	\end{tikzpicture}
	\caption{ The ring $\M_2=H^{*,*}(pt; \underline{\Z/2})$.}
	\label{fig:point}
\end{figure}
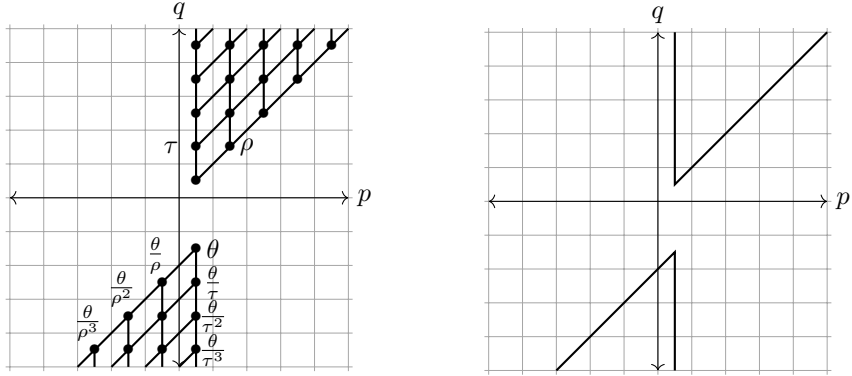

We will often refer to the portion of the cohomology in the first quadrant as the ``top cone" and refer to the other portion as the ``bottom cone". The top cone is polynomial in the elements $\rho$ and $\tau$, where $\rho$ is in bidegree $(1,1)$ and $\tau$ is in bidegree $(0,1)$. Multiplication by $\tau$ is indicated with vertical lines, and multiplication by $\rho$ is indicated with diagonal lines. For example, the nonzero element in $(1,4)$ is equal to $\rho\tau^3$ which is equal to $\tau^3\rho$. The bottom cone is slightly more complicated. The nonzero element $\theta$ in bidegree $(0,-2)$ is divisible by all nonzero elements in the top cone. Explicitly, this means for pairs $i\geq 0$, $j\geq 0$, there exists an element denoted by $\frac{\theta}{\rho^{i}\tau^{j}}$ that satisfies $\rho^i\tau^j\cdot\frac{\theta}{\rho^{i}\tau^{j}}=\theta$. Note $\rho$ and $\tau$ are not invertible elements in the ring; the notation $\frac{\theta}{\tau}$, $\frac{\theta}{\rho}$ is simply used to keep track of how $\rho$ and $\tau$ multiply with these elements. 

While doing computations, it is often easier to work with an abbreviated picture, which is given on the right-hand grid in the above figure. It is understood that there is a $\Z/2$ at each spot within the top cone and within the bottom cone with the relations described above. \medskip

We also include the cohomology of the free orbit $C_2$. As a ring, $H^{*,*}(C_2;\underline{\Z/2})$ is isomorphic to $\Z/2[u,u^{-1}]$ where $u$ is in bidegree $(0,1)$. As an $\M_2$-module, $H^{*,*}(C_2)$ is isomorphic to $\tau^{-1}\M_2/(\rho)$. See Figure \ref{fig:free} for the pictorial representation of this module and its abbreviated version. In these module pictures, action by $\tau$ is indicated by vertical lines, while action by $\rho$ is indicated by diagonal lines. 

\begin{figure}[ht]
	\begin{tikzpicture}[scale=.45]
		\draw[help lines,light-gray] (-5.125,-5.125) grid (5.125, 5.125);
		\draw[<->] (-5,0)--(5,0)node[right]{$p$};
		\draw[<->] (0,-5)--(0,5)node[above]{$q$};
		\anti{0}{0}{black};
		\foreach \y in {-5,...,4}
			\draw (0.5,\y+.5) node{\small{$\bullet$}};
	\end{tikzpicture}\hspace{0.5 in}
	\begin{tikzpicture}[scale=.45]
		\draw[help lines,light-gray] (-5.125,-5.125) grid (5.125, 5.125);
		\draw[<->] (-5,0)--(5,0)node[right]{$p$};
		\draw[<->] (0,-5)--(0,5)node[above]{$q$};
		\anti{0}{0}{black};
	\end{tikzpicture}
	\caption{The cohomology of $C_2$ as an $\M_2$-module.}
	\label{fig:free}
\end{figure}
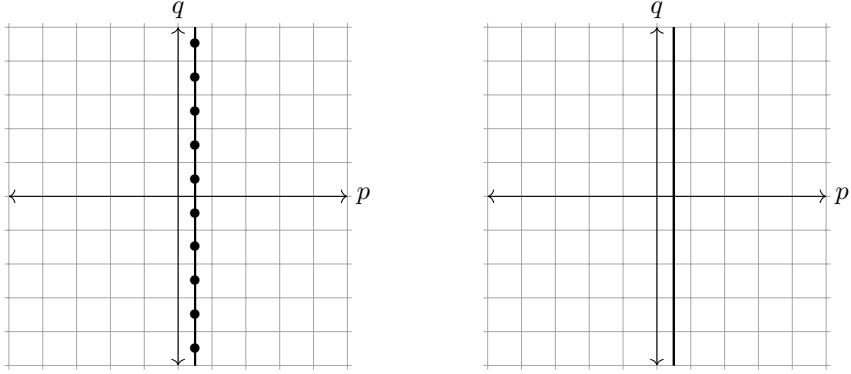 
%
%
\section{Some Computational Tools}\label{ch:comptools}
In this section, we introduce various computational tools that will be used throughout the paper. The first two lemmas relate the Bredon cohomology to the singular cohomology of the quotient space and of the underlying space. We next show how these lemmas can be used to compute the cohomology of the antipodal spheres. After this example, we introduce a lemma that relates the Bredon cohomology to the cohomology of the fixed set via localization. Finally, we end with a few general theorems about the cohomology of finite $C_2$-CW complexes and general $C_2$-manifolds in the discussed coefficient system.\medskip

The first lemma holds for any constant Mackey functor, and it will be extremely useful in starting computations. We state it below for constant $\Z/2$-coefficients.
\begin{lemma}\label{quotient}(The quotient lemma). Let $X$ be a finite $C_2$-CW complex. We have the following isomorphisms for all $p$:
\[ H^{p,0}(X;\underline{\Z/2}) \cong H^{p,0}(X/C_2;\underline{\Z/2})\cong H^{p}_{sing}(X/C_2;\Z/2).\]
\end{lemma}
\begin{proof} We have a quotient map $X\to X/C_2$ that induces a map on cohomology $H^{p,0}(X/C_2;\underline{\Z/2})\to H^{p,0}(X;\underline{\Z/2})$. Note both $H^{p,0}(X;\underline{\Z/2})$ and $H^{p,0}(X/C_2;\underline{\Z/2})$ are integer-graded cohomology theories. These two cohomology theories agree on both orbits $C_2$ and $C_2/C_2$ because the coefficients are given by a constant Mackey functor, and thus the first isomorphism follows for any finite $C_2$-CW complex. The second isomorphism follows because $X/C_2$ is a trivial $C_2$-space. \end{proof}

\begin{lemma}\label{rholocal}($\rho$-localization). Let $X$ be a finite $C_2$-CW complex. Then 
	\[
	\rho^{-1}H^{*,*}(X;\underline{\Z/2}) \cong \rho^{-1}H^{*,*}(X^{C_2};\underline{\Z/2}) \cong \rho^{-1}\M_2\otimes_{\Z/2} H^*_{sing}(X^{C_2};\Z/2).
	\]
\end{lemma}

The proof of this lemma is similar to the proof of the quotient lemma and can be found in \cite{M1}.

For the next lemma, consider the cofiber sequence
	\begin{equation} 
		S^{0,0}\hookrightarrow S^{1,1} \to C_{2+}\wedge S^{1,0}.
	\end{equation}
Smashing with any pointed $C_2$-space $X$, we obtain the cofiber sequence
	\begin{equation}\label{forget}
		S^{0,0}\wedge X \hookrightarrow S^{1,1} \wedge X \to C_{2+}\wedge S^{1,0} \wedge X. 
	\end{equation}
The long exact sequence induced by this cofiber sequence relates multiplication by the element $\rho$ to the singular cohomology of the space. Specifically, the long exact sequence is as stated in the following lemma. This statement can be found in \cite{K1} and is originally due to \cite{AM}.

\begin{lemma}\label{forgetfulles}(The forgetful long exact sequence). 
Let $X$ be a pointed $C_2$-space. For every integer $q$, we have a long exact sequence
\begin{center}
	\begin{tikzcd}
		\arrow[r] & \tilde{H}^{p-1,q}(X) \arrow[r,"\rho \cdot"]& \tilde{H}^{p,q+1}(X) \arrow[r,"\psi"] & \tilde{H}^{p}_{sing}(X) \arrow[r] & \tilde{H}^{p,q}(X) \arrow[r] & ~
	\end{tikzcd}
\end{center}
where the coefficients are understood to be $\underline{\Z/2}$. 
\end{lemma}

We will refer to the map $\psi:\H^{p,q}(X) \to \H^{p}_{sing}(X)$ as the ``forgetful map". Note in $\M_2$, the element $\tau$ forgets to $1\in H^{0}_{sing}(pt)$, while $\rho$ forgets to zero. Indeed, by the exactness of the forgetful long exact sequence, for any $X$, a given cohomology class forgets to zero if and only if it is the image of $\rho$. 

\begin{example}\label{anticomp} 
Let's see how these tools can be used to compute the cohomology of a $C_2$-space. Note this computation is certainly not a new computation, but instead is done to review a standard fact, as well as to show the reader how we will use the computational tools discussed in this section. 

Let $S^n_a$ denote the equivariant $n$-sphere whose $C_2$-action is given by the antipodal map. We proceed by induction to show \[H^{*,*}(S^n_a;\underline{\Z/2}) \cong\tau^{-1}\M_2/(\rho^{n+1})\]
as an $\M_2$-module. The module $\tau^{-1}\M_2/(\rho^{n+1})$ can be represented pictorially as shown in Figure \ref{fig:anti}. 
\begin{figure}[ht]
	\begin{tikzpicture}[scale=.45]
		\draw[help lines,light-gray] (-1.125,-5.125) grid (7.125, 5.125);
		\draw[<->] (-1,0)--(7,0)node[right]{$p$};
		\draw[<->] (0,-5)--(0,5)node[above]{$q$};
		\foreach \x in {0,...,5}
			\draw[thick] (\x+1/2, -5)--(\x+1/2, 5);
		\draw[thick] (1/2,-4.5)--(5.5,.5);
		\draw[thick] (1/2,-3.5)--(5.5,1.5);
		\draw[thick] (1/2,-2.5)--(5.5,2.5);
		\draw[thick] (1/2,-1.5)--(5.5,3.5);
		\draw[thick] (1/2,-.5)--(5.5,4.5);
		\draw[thick] (1/2, .5)--(5,5);
		\draw[thick] (1/2, 1.5)--(4,5);
		\draw[thick] (1/2, 2.5)--(3,5);
		\draw[thick] (1/2, 3.5)--(2,5);
		\draw[thick] (1/2, 4.5)--(1,5);
		\draw[thick] (5.5, -.5)--(1,-5);
		\draw[thick] (5.5, -1.5)--(2,-5);
		\draw[thick] (5.5, -2.5)--(3,-5);
		\draw[thick] (5.5, -3.5)--(4,-5);
		\draw[thick] (5.5, -4.5)--(5,-5);
		\foreach \x in {0,...,5}
			{
			\foreach \y in {-5,-4,-3,-2,-1,0,1,2,3,4}
				\draw (\x+.5,\y+.5) node{\small{$\bullet$}};
			}
\end{tikzpicture}
	\caption{The $\M_2$-module $\tau^{-1}\M_2/(\rho^{n+1})$ when $n=5$.}
	\label{fig:anti}
\end{figure}
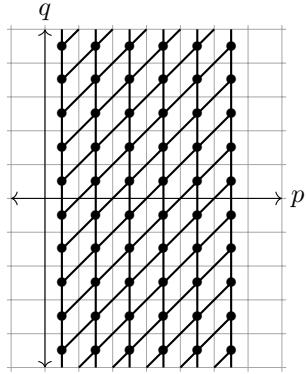

The base case in our inductive argument is given by $n=0$, where $S^0_a$ is understood to be the free orbit $C_2$. This case is done by the comments made in Section \ref{ch:preliminaries}. For the inductive hypothesis, let $n\geq 1$ and suppose $H^{*,*}(S^{n-1}_a)\cong \tau^{-1}\M_2/(\rho^{n})$. 

Our goal is to compute $H^{*,*}(S^n_a)$ using the inductive hypothesis, so we need a way to relate the cohomology of the $(n-1)$-dimensional antipodal sphere to the cohomology of the $n$-dimensional antipodal sphere. Consider the cofiber sequence
	\[
	S^{n-1}_{a+}\hookrightarrow S^{n}_{a+} \to S^{n,0}\wedge C_{2+}
	\] 
where $S^{n-1}_{a}$ includes as the equator of $S^{n}_a$. (Note the disjoint basepoint on the antipodal spheres is needed to run the Puppe sequence, as described in Section \ref{ch:preliminaries}.) The quotient space $S^n_a/S^{n-1}_a$ is nonequivariantly homeomorphic to $S^n \vee S^n$ and the inherited $C_2$-action swaps the two copies; this is exactly the $C_2$-space $S^{n,0}\wedge C_{2+}$. For every integer $q$, we have a long exact sequence given by
	\begin{align*}
 		\overset{d^{p-1,q}}{\to} \tilde{H}^{p,q}(S^{n,0}\wedge C_{2+})\to \H^{p,q}(S^{n}_{a+}) \to \H^{p,q}(S^{n-1}_{a+}) \overset{d^{p,q}}{\to} \tilde{H}^{p+1,q}(S^{n,0}\wedge C_{2+})\to
	\end{align*}
Note $\H^{*,*}(X_+)=H^{*,*}(X)$, so we can express this long exact sequence as
	\begin{align}\label{eq:exactness}
 		\overset{d^{p-1,q}}{\to} \tilde{H}^{p,q}(S^{n,0}\wedge C_{2+})\to H^{p,q}(S^{n}_{a}) \to H^{p,q}(S^{n-1}_{a}) \overset{d^{p,q}}{\to} \tilde{H}^{p+1,q}(S^{n,0}\wedge C_{2+})\to
	\end{align}

In order to find $H^{p,q}(S^n_a)$, we need to understand the differentials 
	\[
	d^{p,q}: H^{p,q}(S^{n-1}_{a})\to \tilde{H}^{p+1,q}(C_{2+}\wedge S^{n,0})
	\] 
for all $(p,q)$. It is helpful to consider all of these differentials at once. Let 
	\[
	d=\underset{p,q}{\oplus}d^{p,q}:H^{*,*}(S^{n-1}_a)\to H^{*+1,*}(S^{n,0}\wedge C_{2+})
	\] 
be the total differential. Note this differential is a module map, i.e. for every $r \in H^{*,*}(pt)$ and class $\alpha \in H^{*,*}(S^{n-1}_a)$, $d(r\alpha)=rd(\alpha)$. 

By the inductive hypothesis, we already understand the cohomology of the domain. Namely $H^{*,*}(S^{n-1}_a)\cong \tau^{-1}\M_2/(\rho^{n})$. Observe the codomain is the cohomology of the space $S^{n,0}\wedge C_{2+} =\Sigma^{n,0}C_{2+}$ which by the suspension isomorphism is given by
	\[
	\tilde{H}^{*,*}(\Sigma^{n,0}C_{2+})\cong \tilde{H}^{*-n,*}(C_{2+}) = \Sigma^{n,0}H^{*,*}(C_{2}).
	\]
Pictorially, the reduced cohomology of $C_{2+}\wedge S^{n,0}$ is just given by taking the cohomology of $C_2$ and shifting it to the right $n$ units.

It is helpful to illustrate $d$ via the following color-coded picture. 
\begin{figure}[ht]
	\begin{tikzpicture}[scale=.6]
		\draw[help lines,light-gray] (-1.125,-3.125) grid (7.125, 3.125);
		\draw[<->] (-1,0)--(7,0)node[right]{$p$};
		\draw[<->] (0,-3)--(0,3)node[above]{$q$};
		\foreach \x in {0,...,5}
			\draw[thick,red] (\x+1/2, -3)--(\x+1/2, 3);
		\draw[thick,red] (2,-3)--(5.5,.5);
		\draw[thick,red] (1,-3)--(5.5,1.5);
		\draw[thick, red] (1/2,-2.5)--(5.5,2.5);
		\draw[thick, red] (1/2,-1.5)--(5,3);
		\draw[thick, red] (1/2,-.5)--(4,3);
		\draw[thick, red] (1/2, .5)--(3,3);
		\draw[thick,red] (1/2, 1.5)--(2,3);
		\draw[thick,red] (1/2, 2.5)--(1,3);
		\draw[thick,red] (5.5, -.5)--(3,-3);
		\draw[thick,red] (5.5, -1.5)--(4,-3);
		\draw[thick,red] (5.5, -2.5)--(5,-3);
		\foreach \x in {0,...,5}
			{
			\foreach \y in {-3,-2,-1,0,1,2}
			\draw[red] (\x+.5,\y+.5) node{$\bullet$};
			};
		\foreach \y in {-3,-2,-1,0,1,2}
			\draw[blue] (6.5,\y+.5) node{$\bullet$};
		\draw[thick,blue] (6.5,-3)--(6.5,3);
		\foreach \y in {-3,-2,-1,0,1,2}
			\draw[thick, black,->] (5.55,\y+.5)--(6.45,\y+.5);
		\draw (6,.07)--(6,-.07) node[below]{$n$};
		\draw (6.5,3.5) node{$d^{n-1,q}$};
	\end{tikzpicture} 
	\caption{The differential $d:${\color{red}{ $H^{*,*}(S^{n-1}_{a} )$}} $\to$ {\color{blue}{$\tilde{H}^{*+1,*}(S^{n,0}\wedge C_{2+})$}}.}
	\label{fig:anti2}
\end{figure}
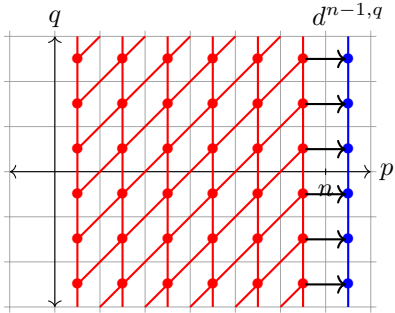
The only possible nonzero differentials occur in topological dimension $(n-1$). Since $d$ is a module map, it commutes with action of $\tau$. The element $\tau$ acts on both modules as an isomorphism, so it suffices to find $d^{n-1,q}$ for a single $q$. Let's consider $q=0$. By the quotient lemma given in Lemma \ref{quotient}, we have the following isomorphisms
	\[
	H^{n,0}(S^n_a) \cong H^n_{sing}(S^n_a/C_2) \cong H^n_{sing}(\R P^{n}) \cong \Z/2.
	\] 
By the exactness of the sequence in \ref{eq:exactness}, we have the short exact sequence
	\[
	0\to \coker(d^{n-1,0}) \to H^{n,0}(S^n_a)\to \ker(d^{n,0})\to 0. 
	\]
Now $H^{n,0}(S^{n-1}_a)=0$ so $\ker(d^{n,0})=0$, and it must be that $\coker(d^{n-1,0})\cong \Z/2$ and $d^{n-1,0}=0$. Thus the total differential $d$ is zero. 

We now need to solve the extension problem of $\M_2$-modules
	\[
	0 \to {\color{blue}{\coker(d)}} \to H^{*,*}(S^n_a) \to {\color{red}{\ker(d)}} \to 0.
	\]
The kernel and cokernel are illustrated below (they are just the domain and codomain of $d$, respectively). Note the extension could be trivial, i.e. $H^{*,*}(S^{n}_a)\cong \coker(d)\oplus \ker(d)$ as $\M_2$-modules, or there could be elements of $\ker(d)$ in topological dimension $(n-1)$ that lift to elements of $H^{n-1,0}(S^n_a)$ with a nontrivial $\rho$-action, as illustrated below. 

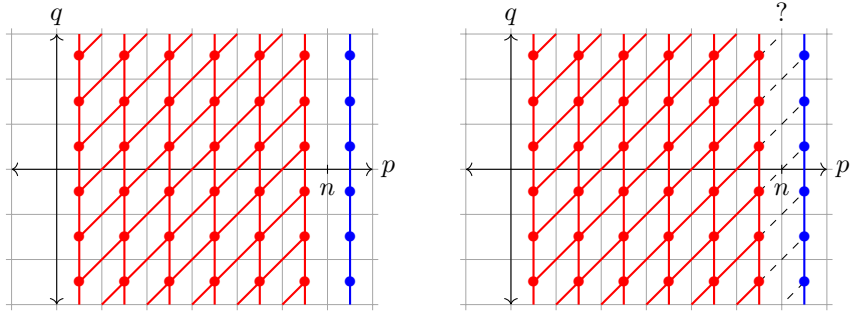
\begin{figure}[ht]
	\begin{tikzpicture}[scale=.6]
		\draw[help lines,light-gray] (-1.125,-3.125) grid (7.125, 3.125);
		\draw[<->] (-1,0)--(7,0)node[right]{$p$};
		\draw[<->] (0,-3)--(0,3)node[above]{$q$};
		\foreach \x in {0,...,5}
			\draw[thick,red] (\x+1/2, -3)--(\x+1/2, 3);
		\draw[thick,red] (2,-3)--(5.5,.5);
		\draw[thick,red] (1,-3)--(5.5,1.5);
		\draw[thick, red] (1/2,-2.5)--(5.5,2.5);
		\draw[thick, red] (1/2,-1.5)--(5,3);
		\draw[thick, red] (1/2,-.5)--(4,3);
		\draw[thick, red] (1/2, .5)--(3,3);
		\draw[thick,red] (1/2, 1.5)--(2,3);
		\draw[thick,red] (1/2, 2.5)--(1,3);
		\draw[thick,red] (5.5, -.5)--(3,-3);
		\draw[thick,red] (5.5, -1.5)--(4,-3);
		\draw[thick,red] (5.5, -2.5)--(5,-3);
		\foreach \x in {0,...,5}
			{
			\foreach \y in {-3,-2,-1,0,1,2}
				\draw[red] (\x+.5,\y+.5) node{$\bullet$};
			};
		\foreach \y in {-3,-2,-1,0,1,2}
			\draw[blue] (6.5,\y+.5) node{$\bullet$};
		\draw[thick,blue] (6.5,-3)--(6.5,3);
		\draw (6,.1)--(6,-.1) node[below]{$n$};
	\end{tikzpicture} \hspace{0.2 in}
	\begin{tikzpicture}[scale=.6]
		\draw[help lines,light-gray] (-1.125,-3.125) grid (7.125, 3.125);
		\draw[<->] (-1,0)--(7,0)node[right]{$p$};
		\draw[<->] (0,-3)--(0,3)node[above]{$q$};
		\foreach \y in {-3,-2,-1,0,1}
			\draw[dashed] (5.5, \y+0.5)--(6.5,\y+1.5);
		\draw[dashed] (5.5, 2.5)--(6,3);
		\draw[dashed] (6.5,-2.5)--(6,-3);
		\foreach \x in {0,...,5}
			\draw[thick,red] (\x+1/2, -3)--(\x+1/2, 3);
		\draw[thick,red] (2,-3)--(5.5,.5);
		\draw[thick,red] (1,-3)--(5.5,1.5);
		\draw[thick, red] (1/2,-2.5)--(5.5,2.5);
		\draw[thick, red] (1/2,-1.5)--(5,3);
		\draw[thick, red] (1/2,-.5)--(4,3);
		\draw[thick, red] (1/2, .5)--(3,3);
		\draw[thick,red] (1/2, 1.5)--(2,3);
		\draw[thick,red] (1/2, 2.5)--(1,3);
		\draw[thick,red] (5.5, -.5)--(3,-3);
		\draw[thick,red] (5.5, -1.5)--(4,-3);
		\draw[thick,red] (5.5, -2.5)--(5,-3);
		\foreach \x in {0,...,5}
			{
			\foreach \y in {-3,-2,-1,0,1,2}
				\draw[red] (\x+.5,\y+.5) node{$\bullet$};
			};
		\foreach \y in {-3,-2,-1,0,1,2}
			\draw[blue] (6.5,\y+.5) node{$\bullet$};
		\draw[thick,blue] (6.5,-3)--(6.5,3);
		\draw (6,.1)--(6,-.1) node[below]{$n$};
		\draw (6,3.5) node{?};
	\end{tikzpicture} 
	\caption{The extension problem.}
	\label{fig:anti3}
\end{figure}

We use a portion of the forgetful long exact sequence to see the extension is indeed nontrivial. Consider
\begin{center}
	\begin{tikzcd}[column sep=small]
		\tilde{H}^{n-1,q-1}(S^{n}_{a+}) \arrow[r,"\cdot\rho"] & \tilde{H}^{n,q}(S^{n}_{a+}) \arrow[r]& \tilde{H}^{n}_{sing}(S^n_{a+}) \arrow[r] & \tilde{H}^{n,q-1}(S^n_{a+}) \arrow[r,"\cdot\rho"] & \tilde{H}^{n+1,q}(S^n_{a+})
	\end{tikzcd}
\end{center}
Observe $ \tilde{H}^{n+1,q}(S^n_{a+}) = 0$ while $\tilde{H}^{n,q-1}(S^n_{a+})\cong \Z/2 \cong \tilde{H}^{n}_{sing}(S^n_{a+})$. Exactness then shows the multiplication by $\rho$ in the left map must be an isomorphism for all $q$, and we conclude the module structure is the one shown in Figure \ref{fig:anti}.
\end{example}

\begin{rmk} 
We will be making arguments such as the one in Example \ref{anticomp} throughout the paper. We will be less verbose and will use abbreviated pictures in future computations. For example, Figures \ref{fig:anti2} and \ref{fig:anti3} would be combined into the single abbreviated figure shown below. If the reader gets confused about the techniques in a future computation, we invite them to return to the above example as a kind of computational tutorial.
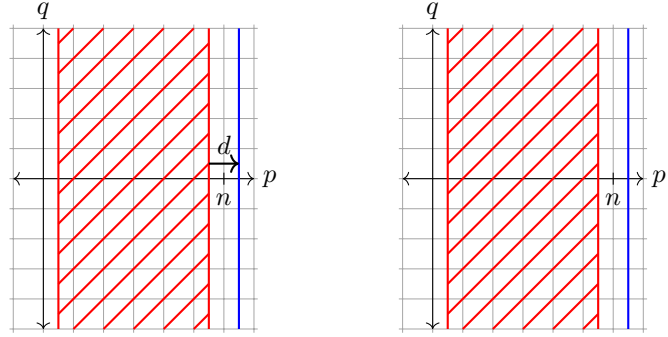
\begin{figure}[ht]
	\begin{tikzpicture}[scale=.4]
		\draw[help lines,light-gray] (-1.125,-5.125) grid (7.125, 5.125);
		\draw[<->] (-1,0)--(7,0)node[right]{$p$};
		\draw[<->] (0,-5)--(0,5)node[above]{$q$};
		\draw[thick,red] (.5,-5)--(.5,5);
		\draw[thick,red] (5.5,-5)--(5.5,5);
		\foreach \y in {-5,-4,-3,-2,-1}
			\draw[thick,red] (.5,\y+.5) --(5.5,\y+5.5);
		\draw[thick, red] (.5,.5)--(5,5);
		\draw[thick, red] (.5,1.5)--(4,5);
		\draw[thick, red] (.5,2.5)--(3,5);
		\draw[thick, red] (.5,3.5)--(2,5);
		\draw[thick, red] (.5,4.5)--(1,5);
		\draw[thick, red] (5.5,-.5)--(1,-5);
		\draw[thick, red] (5.5,-1.5)--(2,-5);
		\draw[thick, red] (5.5,-2.5)--(3,-5);
		\draw[thick, red] (5.5,-3.5)--(4,-5);
		\draw[thick, red] (5.5,-4.5)--(5,-5);
		\draw[thick,blue] (6.5,-5)--(6.5,5);
		\draw[thick, black,->] (5.5,.5)--node[above]{$d$}(6.5,.5);
		\draw (6,.2)--(6,-.2) node[below]{$n$};
	\end{tikzpicture} \hspace{0.5 in} 
	\begin{tikzpicture}[scale=.4]
		\draw[help lines,light-gray] (-1.125,-5.125) grid (7.125, 5.125);
		\draw[<->] (-1,0)--(7,0)node[right]{$p$};
		\draw[<->] (0,-5)--(0,5)node[above]{$q$};
		\draw[thick,red] (.5,-5)--(.5,5);
		\draw[thick,red] (5.5,-5)--(5.5,5);
		\foreach \y in {-5,-4,-3,-2,-1}
			\draw[thick,red] (.5,\y+.5) --(5.5,\y+5.5);
		\draw[thick, red] (.5,.5)--(5,5);
		\draw[thick, red] (.5,1.5)--(4,5);
		\draw[thick, red] (.5,2.5)--(3,5);
		\draw[thick, red] (.5,3.5)--(2,5);
		\draw[thick, red] (.5,4.5)--(1,5);
		\draw[thick, red] (5.5,-.5)--(1,-5);
		\draw[thick, red] (5.5,-1.5)--(2,-5);
		\draw[thick, red] (5.5,-2.5)--(3,-5);
		\draw[thick, red] (5.5,-3.5)--(4,-5);
		\draw[thick, red] (5.5,-4.5)--(5,-5);
		\draw[thick,blue] (6.5,-5)--(6.5,5);
		\draw (6,.2)--(6,-.2) node[below]{$n$};
	\end{tikzpicture}
	\caption{The abbreviated pictures for the differential and the extension problem in Example \ref{anticomp}.}
\end{figure}
\end{rmk}
During our computations, we will often encounter spaces of the form $Y\times C_2$ where $Y$ is some finite CW-complex. The cohomology of such a space depends entirely on the singular cohomology of $Y$, as shown in the lemma below. 
\begin{lemma}\label{times} 
Let $Y$ be a $C_2$-space. The cohomology of the free $C_2$-space $C_2\times Y$ is given by
	\[
	H^{*,*}(Y\times C_2; \underline{\Z/2})\cong \Z/2[\tau,\tau^{-1}]\otimes_{\Z/2} H^{*}_{sing}(Y;\Z/2)
	\]
as $\M_2$-modules. If $Y$ is a based $C_2$-space $Y$, we also have the module isomorphism
	\[
	\H^{*,*}(Y\wedge C_{2+};\underline{\Z/2}) \cong \Z/2[\tau,\tau^{-1}]\otimes_{\Z/2} \H^{*}_{sing}(Y;\Z/2).
	\]
\end{lemma}
\begin{proof}
From the results in \cite{dS}, a model representing Bredon cohomology in constant $\underline{\Z/2}$-coefficients is given by $K(\underline{\Z/2};p,q)\simeq \Z/2\langle S^{p,q} \rangle$ where $\Z/2\langle S^{p,q} \rangle$ has underlying space given by the usual Dold-Thom space of configurations of points in $S^p$ with labels in $\Z/2$, and has $C_2$-action given by the action on $S^{p,q}$. 

Any $C_2$-equivariant map $Y\times C_2\to \Z/2\langle S^{p,q} \rangle$ is entirely determined by the restriction to $Y\times \{1\}$ and thus
\begin{align*}
	H^{p,q}(Y\times C_2;\underline{\Z/2}) &\cong [Y\times C_2,\Z/2\langle S^{p,q} \rangle]_{C_2}\\
	&\cong[Y,\Z/2\langle S^p\rangle]_e\\
	&\cong H^{p}_{sing}(Y;\Z/2)
\end{align*}
where $[-,-]_{C_2}$ denotes the collection of $C_2$-equivariant maps up to $C_2$-equivariant homotopy, and $[-,-]_e$ denotes the collection of nonequivariant homotopy classes of maps. This establishes the isomorphism stated in the lemma as bigraded abelian groups. 

For the module structure, note the above shows the forgetful map $\psi:H^{p,q}(Y\times C_2)\to H^{p}_{sing}(Y\times C_2)$ is just the diagonal map, so in particular, nothing is in the kernel of $\psi$ and $\rho$ must act trivially. On the other hand $\psi(\tau)=1$, so $\tau$ must act as an isomorphism on $H^{*,*}(Y\times C_2)$. This shows 
	\[
	H^{*,*}(Y\times C_2; \underline{\Z/2})\cong \Z/2[\tau,\tau^{-1}]\otimes_{\Z/2} H^{*}_{sing}(Y;\Z/2).
	\]
The reduced statement is proven similarly. 
\end{proof}

If $X$ is a trivial $C_2$-space, then we can similarly state the Bredon cohomology of $X$ entirely in terms of the singular cohomology of $X$.
\begin{lemma}\label{trivial}
Let $X$ be a trivial finite $C_2$-CW complex. Then 
	\[
	H^{*,*}(X;\underline{\Z/2})\cong \M_2\otimes_{\Z/2} H^{*}_{sing}(X;\Z/2).
	\]
\end{lemma}
\begin{proof}
We have a functor $\Psi: \Top \to C_2$-$\Top$ that takes a space $X$ and regards it as a trivial $C_2$-space. For each integer $q$ we can define two cohomology theories on $\Top$ by $(\M_2\otimes_{\Z/2}H^{*}_{sing}(X;\Z/2))^q$ and $H^{*,q}(\Psi(X);\underline{\Z/2})$. As explained in the previous proof,
	\[
	H^{p}_{sing}(X;\Z/2) = [X,\Z/2\langle S^{p} \rangle]_e
	\]
while 
	\[
	H^{p,0}(\Psi(X);\underline{\Z/2}) = [\Psi(X), \Z/2\langle S^{p,0}\rangle]_{C_2}.
	\]
Since both $\Psi(X)$ and $ \Z/2\langle S^{p,0}\rangle$ are trivial $C_2$-spaces, the above is exactly equal to $[X,\Z/2\langle S^{p} \rangle]_e$. Thus we have a clear map
	\[
	H^{*}_{sing}(X;\Z/2) = H^{*,0}(\Psi(X);\underline{\Z/2})\hookrightarrow H^{*,*}(\Psi(X);\underline{\Z/2})
	\]
which induces a map from the free module
	\[
	\M_2\otimes_{\Z/2}H^{*}_{sing}(X;\Z/2) \to H^{*,*}(\Psi(X);\underline{\Z/2}).
	\]
Restricting to the $q$-th grading, we have a map between cohomology theories
	\[ 
	(\M_2\otimes_{\Z/2}H^{*}_{sing}(X;\Z/2))^{q} \to H^{*,q}(\Psi(X);\underline{\Z/2}).
	\]
This map is an isomorphism when $X=pt$, and so these cohomology theories agree for finite CW-complexes for all values of $q$. This establishes the stated isomorphism as bigraded abelian groups, and the module structure follows from noting the map $\M_2\otimes_{\Z/2}H^{*}_{sing}(X;\Z/2) \to H^{*,*}(\Psi(X);\underline{\Z/2})$ is actually a module map.
\end{proof}

We next state an important theorem about the cohomology of $C_2$-manifolds. Here by ``$C_2$-manifold" we mean a piecewise linear manifold with a locally linear $C_2$-action (we want to ensure the fixed set is a disjoint union of submanifolds). By closed, we simply mean a closed manifold in the nonequivariant sense. The proof of the following theorem is somewhat tedious, so it has been been moved to Appendix \ref{ch:manfacts}.

\begin{theorem}\label{topm2} Let $X$ be an $n$-dimensional, closed $C_2$-manifold with a nonfree $C_2$-action. Suppose $n-k$ is the largest dimension of submanifold appearing as a component of the fixed set. Then there is exactly one summand of $H^{*,*}(X;\underline{\Z/2})$ of the form $\Sigma^{i,j}\M_2$ where $i\geq n$, and it occurs for $(i,j)=(n,k)$.
\end{theorem}

\subsection{A Structure Theorem}
We conclude this section by recalling a fact about the coefficient ring $\M_2$ as well as a structure theorem for the cohomology of finite $C_2$-CW complexes. The two theorems below can be found in \cite{M1}.

\begin{theorem}[C. May]\label{injective} As a module over itself, $\M_2$ is injective.
\end{theorem}

The next theorem is a precise statement of the decomposition mentioned in the introduction. Note it is necessary that $X$ is a finite $C_2$-CW complex in the sense that it only contains finitely many cells. Any closed $C_2$-manifold can be given the structure of a $C_2$-CW complex, so in particular, this theorem applies to all closed $C_2$-manifolds.

\begin{theorem}[C. May]\label{structure} For any finite $C_2$-CW complex $X$, we can decompose the $RO(C_2)$-graded cohomology of $X$ with constant $\underline{\Z/2}$-coefficients as
	\[
	H^{*,*}(X;\underline{\Z/2}) = (\oplus_i\Sigma^{p_i,q_i}\M_2 ) \oplus (\oplus_j \Sigma^{p_j ,0} A_{n_j} )
	\] 
as a module over $\M_2 = H^{*,*}(pt;\underline{\Z/2})$ where $A_{n}$ denotes the cohomology of the $n$-sphere with the free antipodal action.
\end{theorem}
%
%
\section{Classification of \texorpdfstring{$C_2$}{C2}-surfaces}\label{ch:class}
In \cite{D2}, all $C_2$-surfaces were classified up to equivariant isomorphism, and furthermore, a language was developed for describing the $C_2$-structure on a given equivariant surface. These descriptions are essential for many of the computations given in this paper. In this section, we review the needed terms and notations as well as some of the classification theorems. Notice all proofs are omitted in this section, and we direct the curious reader to \cite{D2}. 

\subsection{Free Actions} 
In order to state the classification of free $C_2$-surfaces, we need to review one construction from equivariant surgery. This definition was first given in the introduction, but we restate it here for reference.
\begin{definition}\label{consumdef} Let $X$ be a nontrivial $C_2$-surface and $Y$ be a nonequivariant surface. We can form the \textbf{equivariant connected sum} of $X$ and $Y$ as follows. Let $Y'$ denote the space obtained by removing a small disk from $Y$. Let $D$ be a disk in $X$ that is disjoint from its conjugate disk $\sigma D$ and let $X'$ denote the space obtained by removing both of these disks. Choose an isomorphism $f:\partial Y' \to \partial D$. Then the equivariant connected sum is given by
\[\left(Y' \times \{0\}\right) \sqcup \left(Y' \times \{1\} \right)\sqcup X']/\sim\]
where $(y,0) \sim f(y)$ and $(y,1) \sim \sigma(f(y))$ for $y\in \partial Y'$. We denote this space by $X \#_2 Y$.
\end{definition}

\begin{rmk}\label{specialconnsum} 
In our classifications, there are three important examples of equivariant connected sums that warrant their own notation. The first occurs when $Y$ is the projective plane. In this case, we refer to the surgery as ``adding dual cross caps" and write $X+[DCC]$ for $X\#_2\R P^2$. The phrase ``dual cross caps" arrises from the pictorial representation often used to denote such surgery; see Figure \ref{fig:dccex} below for an example. Note we can add more than one set of dual cross caps to form $X+[DCC]+[DCC]$ which we will denote $X+2[DCC]$. In general, $X+n[DCC]$ is the space obtained by adding $n$ dual cross caps to the $C_2$-surface $X$.

\begin{figure}[ht]
	\begin{tikzpicture}[scale=0.8]
		\draw[<->] (-.2,1.4)--(.2,1.4);
		\draw (0,0) ellipse (2cm and 1cm);
		\draw (-0.4,0) to[out=330,in=210] (.4,0) ;
		\draw (-.35,-.02) to[out=35,in=145] (.35,-.02);
		\draw[very thick,blue] (0,.55) ellipse (.1cm and 0.45cm);
		\draw[white,fill=white] (-.3,0.15) rectangle (0,.96);
		\draw[very thick,blue,dashed] (0,.55) ellipse (.1cm and 0.45cm);
		\draw[very thick,blue] (0,-.56) ellipse (.1cm and 0.43cm);
		\draw[white,fill=white] (-.3,-0.95) rectangle (0,-.14);
		\draw[very thick,blue,dashed] (0,-.56) ellipse (.1cm and 0.43cm);
		\draw (0.9,-.3) circle (0.2cm);
		\draw (1,-.2)--(.8,-.4);
		\draw (1,-.4)--(.8,-.2);
		\draw (-0.9,-.3) circle (0.2cm);
		\draw (-1,-.2)--(-.8,-.4);
		\draw (-1,-.4)--(-.8,-.2);
	\end{tikzpicture}
	\caption{Dual cross caps added to a genus one torus with a reflection action. The fixed set is shown in \textbf{{\color{blue}{blue}}}.}
	\label{fig:dccex}
\end{figure}
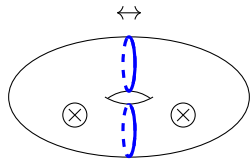

The other two examples occur when $X$ is one of $S^{2,2}$ or $S^{2,1}$, the two nonfree, nontrivial equivariant spheres. We refer to such spaces as \textbf{doubling spaces}, and denote the space $X\#_2 Y$ by Doub$(Y,1:S^{1,1})$ or Doub$(Y,1:S^{1,0})$, respectively. The phrase ``doubling" is used to acknowlege that, nonequivariantly, $X\#_2Y$ is homeomorphic to $Y\# Y$. 
\end{rmk}

The following lemma classifies all free actions on the genus one torus.

\begin{lemma}\label{freetorus} 
There are exactly two free actions on the genus one torus, up to equivariant isomorphism. The first is an orientation reversing action given by embedding the torus in $\R^3$ and restricting the antipodal action; we denote this space by $T_1^{anti}$. The second is an orientation preserving action given by rotating the torus $180^\circ$ around the center hole; we denote this space by $T_1^{rot}$.  
\end{lemma}

Recall from Section \ref{ch:comptools},  we write $S^n_a$ for the $n$-dimensional antipodal sphere. We are now ready to state the main classification theorem for free $C_2$-surfaces.  All statements in the theorem should be followed with ``up to equivariant isomorphism". Note by genus of a nonorientable space $X$ we simply mean the $\beta$-genus of $X$ which was defined in the introduction to be the dimension of $H^1_{sing}(X;\Z/2)$. We denote such a space by $N_s$ where $s$ is the genus.

\begin{theorem}\label{freeclass}(Classification of free actions).

\begin{enumerate}
	\item[(i)] There is exactly one free structure on the even genus torus $T_{2k}$ which is given by $S^2_a\#_2 T_k$.
	\item[(ii)] There are exactly two free structures on the odd genus torus $T_{2k+1}$ which are given by $T_1^{anti}\#_2 T_{k}$ and $T_1^{rot}\#_2 T_k$.
	\item[(iii)] There are no free structures on the odd genus non-orientable space $N_{2k+1}$.
	\item[(iv)] There is exactly one free structure on the genus two non-orientable space $N_2$ which is given by $S^2_a+[DCC]$.
	\item[(v)] There are exactly two free structures on the even genus non-orientable space $N_{2k}$ when $k\geq 2$, which are given by $S^2_a+k[DCC]$ and $T^{anti}_1 + (k-1)[DCC]$.
\end{enumerate}
\end{theorem}

The above completely classifies free $C_2$-surfaces. We now state various classification theorems for the nonfree $C_2$-surfaces.

\subsection{Nonfree Actions} 
For the classification of nonfree $C_2$-surfaces, we need to introduce three other types of equivariant surgery. The first two involve removing conjugate disks in order to attach an equivariant handle. There are two types of handles that can be attached. The first handle is given by $S^{1,1}\times D(\R^{1,1})$ where $D(\R^{1,1})$ is the unit disk in $\R^{1,1}$; we will refer to such a handle as an ``$S^{1,1}-$antitube". The second type of handle is given by $S^{1,0}\times D(\R^{1,1})$; we refer to this handle as an ``$S^{1,0}-$antitube". We give the precise definitions below.

\begin{definition} 
Let $X$ be a nontrivial $C_2$-surface. Form a new space denoted $X+[S^{1,0}-AT]$ as follows. Let $D$ be a disk contained in $X$ that is disjoint from its conjugate disk $\sigma D$. Remove both disks from $X$ and then attach an $S^{1,0}-$antitube. Whenever we construct such a space, we say we have done \textbf{$\mathbf{S^{1,0}-}$surgery}.
\end{definition}

We can similarly define \textbf{$\mathbf{S^{1,1}-}$surgery} by instead attaching an $S^{1,1}-$antitube. 

The third type of surgery involves removing a disk isomorphic to the unit disk in $\R^{2,2}$ and sewing in an equivariant M\"obius band. This equivariant M\"obius band can be formed as follows. Begin with the nonequivariant M\"obius bundle over $S^1$, and then define an action on the fibers by reflection. In other words, each fiber should be isomorphic to $\R^{1,1}$ (in particular, the zero section is fixed). If we now take the closed unit disk bundle, note the boundary is a copy of $S^1_a$, as is the boundary of the removed disk $D(\R^{2,2})$. An illustration of this M\"obius bundle is shown below. Conjugate points are indicated by matching symbols, while the fixed set is shown in \textbf{{\color{blue}{blue}}}. 
\begin{figure}[ht]
	\includegraphics[scale=0.6]{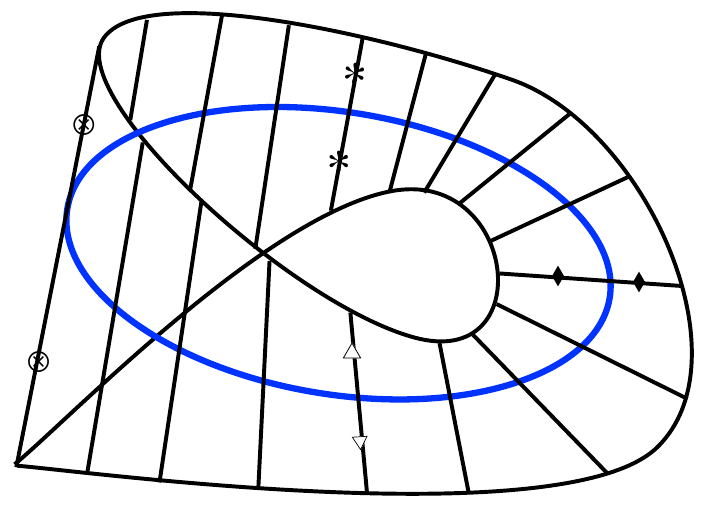}
\end{figure}

\begin{definition} 
Let $X$ be a nontrivial $C_2$-surface that contains an isolated fixed point $p$. Then there exists an open disk $p\in D\subset X$ such that $D\cong D(\R^{2,2})$. Remove $D$ and note $\partial D \cong S^{1}_a$. Now sew in a copy of the M\"obius band described above. We denote this new space by $X+[FM]$ and refer to this surgery as \textbf{$\mathbf{FM-}$surgery}. (Note ``FM" is an abbreviation for ``fixed point to M\"obius band".)
\end{definition}

The complete list of nonfree, nontrivial $C_2$-surfaces is given in \cite{D2}, but it turns out we only need the following takeaway from this list in order to do computations in $\underline{\Z/2}$-coefficients.

\begin{theorem}\label{nonfreeclass} 
Let $X$ be a nonfree, nontrivial surface. If $X$ is not isomorphic to a doubling space or to an equivariant sphere, then $X$ can be obtained by doing $S^{1,0}-$, $S^{1,1}-$, or $FM-$ surgery to an equivariant space of lower $\beta$-genus. 
\end{theorem}
%
%
\section{Computation for Free Actions}\label{ch:freecomp}
In this section, we compute the cohomology of all free $C_2$-surfaces in $\underline{\Z/2}$-coefficients. We first compute the cohomology of the two free tori $T_1^{anti}$ and $T_{1}^{rot}$. We then utilize the decompositions given in Theorem \ref{freeclass} together with these initial computations to compute the cohomology of all free $C_2$-surfaces.

\begin{notation} 
In this section, all coefficients will be understood to be $\underline{\Z/2}$.
\end{notation}

\begin{prop}\label{freetoriprop} 
We have the following isomorphisms of $\M_2$-modules 
	\[
	H^{*,*}(T_1^{anti}) \cong H^{*,*}(S^1_a) \oplus\Sigma^{1,0} H^{*,*}(S^1_a) \cong H^{*,*}(T_1^{rot}).
	\]
\end{prop}
\begin{proof} 
Observe another way to define the antipodal torus is by the product $T_1^{anti}=S^{1,1}\times S^1_a$. This gives rise to the cofiber sequence for $X = T_1^{anti}$
	\[
	S^1_{a+}\hookrightarrow X_+ \to S^{1,1} \wedge S^{1}_{a+}
	\]
which is illustrated below in Figure \ref{fig:cofibertorus}. 
\begin{figure}[ht]
	\begin{tikzpicture}
		\draw[red] (-4.5,0) ellipse (1cm and .3cm);
		\draw[red] (-3.5,-.7) node{$+$};
		\draw (0,0) ellipse (2cm and 1cm);
		\draw (-.8,0) arc (190:350:.8cm and 0.3cm);
		\draw[red,dashed] (-.8,-.1) arc (190:350:.8cm and 0.3cm);
		\draw (.7,-.1) arc (10:170:.7cm and 0.3cm);
		\draw[red] (.8,-.1) arc (-10:190:.8cm and 0.3cm);
		\draw[red] (2,-.7) node{$+$};
		\draw[right hook->] (-3,0)--(-2.5,0);
		\draw[->] (2.5,0)--(3,0);
		\draw (5,0) ellipse (1.5 cm and .8 cm);
		\draw (5,-.10) to[out=110,in=10](4.8,.1);
		\draw (5,-.10) to[out=70,in=190](5.2,.1);
		\draw (5,-.10)--(5,.1);
		\draw[red] (5,-.1) node{\dot};
	\end{tikzpicture}
	\caption{The cofiber sequence $S^1_{a+} \hookrightarrow T_{1+}^{anti}\to S^{1,1} \wedge S^1_{a+}$.}
	\label{fig:cofibertorus}
\end{figure}
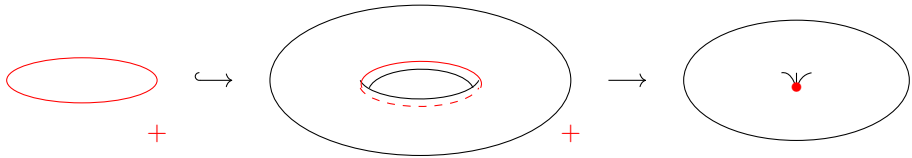
This cofiber sequence gives rise to long exact sequences on cohomology. Similar to Example \ref{anticomp}, we can organize these long exact sequences into the picture shown in Figure \ref{fig:cobfibertorusles} below. By the quotient lemma given in Lemma \ref{quotient}, $\tilde{H}^{0,0}(X_+) \cong \tilde{H}^{0}_{sing}((X/C_2)_+) \cong \Z/2$, and so $d^{0,0}=0$. From the module structure, we conclude $d^{p,q}=0$ for all $p,q$. 
\begin{figure}[ht]
	\begin{tikzpicture}[scale=.45]
		\draw[help lines,light-gray] (-5.125,-5.125) grid (5.125, 5.125);
		\draw[<->] (-5,0)--(5,0)node[right]{$p$};
		\draw[<->] (0,-5)--(0,5)node[above]{$q$};
		\anti{-0.1}{1}{red};
		\anti{1.1}{1}{blue};
		\draw[thick,->] (.4,.5)--(1.6,.5);
		\draw[thick,->] (1.4,1.5)--(2.6,1.5);
	\end{tikzpicture}\hspace{0.5in}
	\begin{tikzpicture}[scale=.45]
		\draw[help lines,light-gray] (-5.125,-5.125) grid (5.125, 5.125);
		\draw[<->] (-5,0)--(5,0)node[right]{$p$};
		\draw[<->] (0,-5)--(0,5)node[above]{$q$};
		\anti{-0.1}{1}{red};
		\anti{1.1}{1}{blue};
	\end{tikzpicture}
	\caption{The differential $d:${\color{red}{ $\tilde{H}^{*,*}(S^1_a) $}} $\to$ {\color{blue}{$\tilde{H}^{*+1,*}(S^{1,1}\wedge S^1_{a+})$}}.}
	\label{fig:cobfibertorusles}
\end{figure}
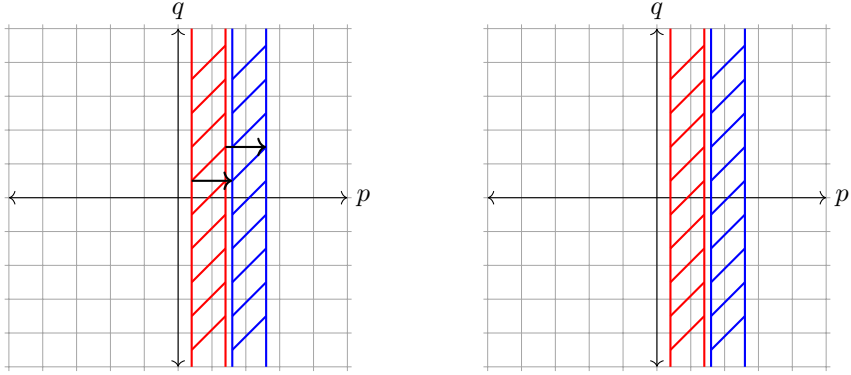

It remains to solve the extension problem
	\[
	0 \to {\color{blue}{\Sigma^{1,1}H^{*,*}(S^1_a)}} \to H^{*,*}(X) \to {\color{red}{H^{*,*}(S^1_a)}} \to 0,
	\]
which is shown on the right in Figure \ref{fig:cobfibertorusles}. The only possibility for a nontrivial extension is for there to exist a class $\alpha\in H^{0,q}(X)$ such that $\rho^2\alpha\in H^{2,q+2}(X)$ is nonzero. 

Let $x\in S^{1,1}$ be one of the two fixed points. We have the following equivariant maps
	\[
	S^1_a\cong S^{1}_a\times \{x\}\hookrightarrow S^1_a\times S^{1,1}=X\overset{\pi_1}{\longrightarrow} S^1_a
	\]
where the last map is just projection onto the first factor. Note this composition is the identity, so in particular, 
	\[
	(\pi_1)^*:H^{*,*}(S^1_a)\to H^{*,*}(X)
	\]
is injective. In fact, this is actually an isomorphism in bidegrees $(0,q)$ because both groups have been computed to be $\Z/2$. Thus for every $\alpha\in H^{0,q}(X)$, there exists a class $\beta\in H^{0,q}(S^1_a)$ such that $\alpha=(\pi_1)^*(\beta)$ and
	\[
	\rho^2\alpha=\rho^2(\pi_1)^*(\beta) = (\pi_1)^*(\rho^2\beta) = 0.
	\]
We conclude the extension in Figure \ref{fig:cobfibertorusles} is trivial, and $H^{*,*}(T_1^{anti}) \cong H^{*,*}(S^1_a) \oplus\Sigma^{1,1} H^{*,*}(S^1_a)$. Lastly, observe as $\M_2$-modules, $\Sigma^{1,1}H^{*,*}(S^1_a) \cong \Sigma^{1,0}H^{*,*}(S^1_a)$.

For the other free torus, note $T_1^{rot}=S^{1,0}\times S^1_a$, so we can make similar use of the cofiber sequence 
	\[
	S^1_{a+} \hookrightarrow T^{rot}_1 \to S^{1,0}\wedge S^1_{a+}
	\]
to see $H^{*,*}(T_1^{rot})\cong H^{*,*}(S^1_a) \oplus\Sigma^{1,0} H^{*,*}(S^1_a)$. We leave the details to the reader.
\end{proof}

We have now computed the cohomology of the two free tori and the free $C_2$-sphere (see Example \ref{anticomp}). By Theorem \ref{freeclass}, all other free $C_2$-surfaces can be obtained by forming equivariant connected sums with these three spaces. We have the following lemmas on how such surgery affects the cohomology. 

\begin{lemma}\label{anticonnlem} 
Suppose $X$ is a free $C_2$-surface. If there is a surface $Y$ such that $X\cong S^{2}_a \#_2 Y$, then 
	\[
	H^{*,*}(X) \cong H^{*,*}(S^2_a) \oplus \left(\Sigma^{1,0}A_0\right)^{\oplus \beta (Y)}.
	\]
\end{lemma}
\begin{proof} 
Let $Y'$ be the space obtained by removing a small disk from $Y$. Consider the cofiber sequence
\begin{equation}\label{cofib1}
	\left(Y'\times C_2\right)_+ \hookrightarrow \left(S^{2}_a \#_2 Y\right)_+ \to \tilde{S^2_a}
\end{equation}
where $\tilde{S^2_a}$ is the ``pinched" space appearing in the cofiber sequence
\begin{equation}\label{cofib2}
	C_{2+} \hookrightarrow S^2_{a+} \to \tilde{S^2_a}.
\end{equation}
An illustration of this cofiber sequence when $Y=T_1$ is shown below in Figure \ref{fig:cofib2ex}. Note the two \textbf{{\color{red}{red}}} points on the right sphere are identified.
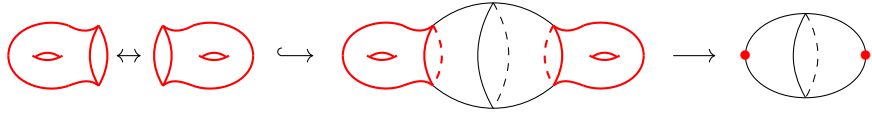
\begin{figure}[ht]
	\begin{tikzpicture}[scale=.8]
		\draw[red,thick] (0,0) to[out=220,in=330] (-.5,0);
		\draw[red,thick] (-.5,0) to[out=160,in=90] (-1.5,-.5);
		\draw[red,thick] (0,-1) to[out=140,in=20] (-.5,-1);
		\draw[red,thick] (-.5,-1) to[out=200,in=270] (-1.5,-.5);
		\draw[red,thick] (-1.1,-.5) to[out=330,in=210] (-.6,-.5) ;
		\draw[red,thick] (-1.05,-.52) to[out=35,in=145] (-.65,-.52);
		\draw[red,thick] (0,0) to[out=240,in=120] (0,-1);
		\draw[red,thick] (0,0) to[out=300,in=60] (0,-1);
		\draw[<->](.3,-.5)--(.7,-.5);
		\draw[white] (0,-1.5);
	\end{tikzpicture}
	\begin{tikzpicture}[scale=.8]
		\draw[red,thick] (2,0) to[out=320,in=210] (2.5,0);
		\draw[red,thick](2.5,0) to[out=20,in=90] (3.5,-.5);
		\draw[red,thick] (2,-1) to[out=40,in=160] (2.5,-1);
		\draw[red,thick] (2.5,-1) to[out=340,in=270] (3.5,-.5);
		\draw[red,thick] (2.6,-.5) to[out=330,in=210] (3.1,-.5);
		\draw[red,thick] (2.65,-.52) to[out=35,in=145] (3.05,-.52);
		\draw[red,thick] (2,0) to[out=240,in=120] (2,-1);
		\draw[red,thick] (2,0) to[out=300,in=60] (2,-1);
		\draw[right hook->] (3.9,-.5)--(4.5,-.5);
		\draw[white] (2,-1.5);
	\end{tikzpicture}\hspace{.1in}
	\begin{tikzpicture}[scale=.8]
		\draw (0,0) to[out=40,in=140] (2,0);
		\draw[red,thick] (2,0) to[out=320,in=210] (2.5,0);
		\draw[red,thick](2.5,0) to[out=20,in=90] (3.5,-.5);
		\draw (0,-1) to[out=320,in=220] (2,-1);
		\draw[red,thick] (2,-1) to[out=40,in=160] (2.5,-1);
		\draw[red,thick] (2.5,-1) to[out=340,in=270] (3.5,-.5);
		\draw[red,thick] (0,0) to[out=220,in=330] (-.5,0);
		\draw[red,thick] (-.5,0) to[out=160,in=90] (-1.5,-.5);
		\draw[red,thick] (0,-1) to[out=140,in=20] (-.5,-1);
		\draw[red,thick] (-.5,-1) to[out=200,in=270] (-1.5,-.5);
		\draw[red,thick] (2.6,-.5) to[out=330,in=210] (3.1,-.5);
		\draw[red,thick] (2.65,-.52) to[out=35,in=145] (3.05,-.52);
		\draw[red,thick] (-1.1,-.5) to[out=330,in=210] (-.6,-.5) ;
		\draw[red,thick] (-1.05,-.52) to[out=35,in=145] (-.65,-.52);
		\draw[red,thick] (0,0) to[out=240,in=120] (0,-1);
		\draw[dashed,red,thick] (0,0) to[out=300,in=60] (0,-1);
		\draw[dashed,red,thick] (2,0) to[out=240,in=120] (2,-1);
		\draw[red,thick] (2,0) to[out=300,in=60] (2,-1);
		\draw (1,.38) to[out=240,in=120] (1,-1.38);
		\draw[dashed] (1,.38) to[out=300,in=60] (1,-1.38);
	\end{tikzpicture}\hspace{.1in}
	\begin{tikzpicture}[scale=.8]
		\draw (0,0) ellipse (1 cm and .7 cm);
		\draw[red] (1,0) node{\dot};
		\draw[red] (-1,0) node{\dot};
		\draw[dashed] (0,.7) to[out=300,in=60] (0,-.7);
		\draw (0,.7) to[out=240,in=120] (0,-.7);
		\draw[->](-2.2,0)--(-1.5,0);
		\draw[white] (0,-1);
	\end{tikzpicture}
	\caption{An example of the cofiber sequence appearing in \ref{cofib1}. Note the red points in the right picture are identified.}
	\label{fig:cofib2ex}
\end{figure}

To make use of the cofiber sequence in \ref{cofib1}, we first compute the cohomology of $\tilde{S^2_a}$. We can extend the cofiber sequence in \ref{cofib2} to 
	\[
	S^{2}_{a+} \to \tilde{S^2_a} \to \Sigma^{1,0}C_{2+}
	\]
and then analyze the connecting homomorphism $d:H^{p,q}(S^2_a)\to \H^{p+1,q}(\Sigma^{1,0}C_{2+})$. This module map $d$ is shown in Figure \ref{fig:cofib2comp}. By the quotient lemma, 
	\[
	\tilde{H}^{0,0}(\tilde{S}^2_a) \cong\tilde{H}^{0}_{sing}(S^2_a/C_2)= \tilde{H}^{0}_{sing}(\R P^2) = 0.
	\] 
Thus $d^{0,0}$ must be an isomorphism, and by the module structure, $d^{0,q}$ must be an isomorphism for all $q$. It follows that $\coker(d)=0$ and $\tilde{H}^{*,*}(\tilde{S}^2_a)$ is the module given on the right in Figure \ref{fig:cofib2comp}. 
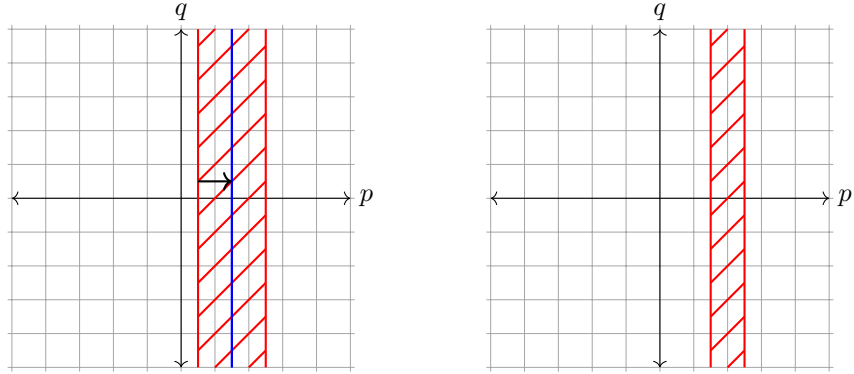
\begin{figure}[ht]
	\begin{tikzpicture}[scale=.45]
		\draw[help lines,light-gray] (-5.125,-5.125) grid (5.125, 5.125);
		\draw[<->] (-5,0)--(5,0)node[right]{$p$};
		\draw[<->] (0,-5)--(0,5)node[above]{$q$};
		\anti{0}{2}{red};
		\draw[red,thick] (.5,3.5)--(2,5);
		\draw[red,thick] (.5,4.5)--(1,5);
		\draw[red,thick] (2.5,-3.5)--(1,-5);
		\draw[red,thick] (2.5,-4.5)--(2,-5);
		\anti{1}{0}{blue};
		\draw[thick, ->] (0.5,.5)--(1.5,.5);
	\end{tikzpicture}\hspace{.5 in}
	\begin{tikzpicture}[scale=.45]
		\draw[help lines,light-gray] (-5.125,-5.125) grid (5.125, 5.125);
		\draw[<->] (-5,0)--(5,0)node[right]{$p$};
		\draw[<->] (0,-5)--(0,5)node[above]{$q$};
		\draw[red,thick] (1.5,4.5)--(2,5);
		\draw[red,thick] (2.5,-4.5)--(2,-5);
		\anti{1}{1}{red};
	\end{tikzpicture}
	\caption{The differential $d:${\color{red}{ $\tilde{H}^{*,*}(S^2_{a+}) $}} $\to$ {\color{blue}{$\tilde{H}^{*+1,*}(\Sigma^{1,0}C_{2+})$}}}
	\label{fig:cofib2comp}
\end{figure}

We now return to the cofiber sequence given in \ref{cofib1}. The cohomology of $Y\times C_2$ is computed by Lemma \ref{times}, namely
	\[
	H^{*,*}(Y\times C_2)\cong \Z/2[\tau,\tau^{-1}]\otimes_{\Z/2}H^{*}_{sing}(Y)\cong A_0\oplus \left(\Sigma^{1,0}A_0 \right)^{\oplus \beta(Y)}.
	\]
The cofiber sequence will thus give rise to the long exact sequences whose differentials are shown in Figure \ref{fig:cofibone}.
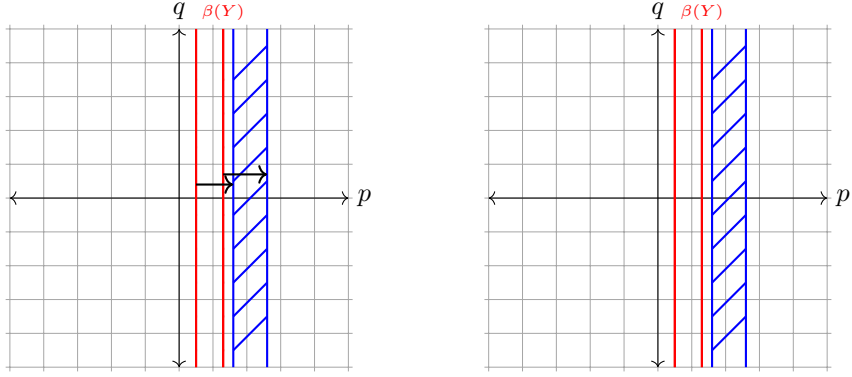
\begin{figure}[ht]
	\begin{tikzpicture}[scale=.45]
		\draw[help lines,light-gray] (-5.125,-5.125) grid (5.125, 5.125);
		\draw[<->] (-5,0)--(5,0)node[right]{$p$};
		\draw[<->] (0,-5)--(0,5)node[above]{$q$};
		\anti{0}{0}{red};
		\anti{.8}{0}{red};
		\lab{.8}{\beta(Y)}{red};
		\anti{1.1}{1}{blue};
		\draw[->,thick](0.5,.4)--(1.6,.4);
		\draw[->,thick](1.3,.7)--(2.6,.7);
	\end{tikzpicture}\hspace{.5 in}
	\begin{tikzpicture}[scale=.45]
		\draw[help lines,light-gray] (-5.125,-5.125) grid (5.125, 5.125);
		\draw[<->] (-5,0)--(5,0)node[right]{$p$};
		\draw[<->] (0,-5)--(0,5)node[above]{$q$};
		\anti{0}{0}{red};
		\anti{.8}{0}{red};
		\lab{.8}{\beta(Y)}{red};
		\anti{1.1}{1}{blue};
	\end{tikzpicture}
	\caption{The differential $d:${\color{red}{ $\tilde{H}^{*,*}((Y'\times C_2)_+) $}} $\to$ {\color{blue}{$\tilde{H}^{*+1,*}(\tilde{S^2_a})$}}}
	\label{fig:cofibone}
\end{figure}

Since $X\cong S^2_a\#_2 Y$, $X/C_2\cong \R P^2\# Y$. We can again use Lemma \ref{quotient} to see
	\[
	H^{0,0}(X)\cong H^{0}_{sing}(X/C_2) \cong \Z/2, ~~\text{ and}
	\]
	\[
	H^{1,0}(X) \cong H^{1}_{sing}(X/C_2) \cong (\Z/2)^{\beta(X)/2} = (\Z/2)^{\beta(Y)+1}.
	\]
Thus $d=0$ and we must solve the extension problem shown on the right in Figure \ref{fig:cofibone}. To do so, consider the following map of cofiber sequences where the left vertical map is collapsing the two nonequivariant components of $Y'\times C_2$ to two points.
\begin{center}
\begin{tikzcd}
	(Y'\times C_{2})_+ \arrow[r,hook] \arrow[d]& X_+ \arrow[r]\arrow[d,"q"]& \tilde{S}^{2}_a\arrow[d]\\
	C_{2+} \arrow[r,hook] & S^{2}_{a+} \arrow[r]& \tilde{S}^{2}_a
\end{tikzcd}
\end{center}
As we've shown, both cofiber sequences induce long exact sequences on cohomology where the differential is zero. Thus we have the following commutative diagram where the rows are exact.
\begin{center}
\begin{tikzcd}
	0\arrow[r]&\tilde{H}^{*,*}(\tilde{S}^{2}_a)\arrow[r] & H^{*,*}(X) \arrow[r]&H^{*,*}(Y'\times C_{2}) \arrow[r]&0\\
	0\arrow[r]&\H^{*,*}(\tilde{S}^{2}_{a}) \arrow[r] \arrow[u,"id"]& H^{*,*}(S^{2}_{a}) \arrow[r]	\arrow[u,"q^*"]& H^{*,*}(C_{2}) \arrow[r]\arrow[u,hook]& 0 
\end{tikzcd}
\end{center}
From exactness, the middle map must be injective. In particular, for all nonzero $\alpha\in H^{0,q}(S^2_a)$, $\rho^2\alpha\neq 0$ and so $\rho^2q^*(\alpha)=q^*(\rho^2\alpha)\neq 0$. We conclude the extension in Figure \ref{fig:cofibone} must be the nontrivial extension given by $H^{*,*}(S^2_a) \oplus \left(\Sigma^{1,0}A_0\right)^{\oplus \beta (Y)}$, as desired. (While this can be seen by making an appropriate choice of basis, it may be faster to apply the structure theorem given in \ref{structure}. The above shows $\rho^2$ induces an isomorphism on bidegrees $(0,q)$, and the only possibility based on this fact and the groups we've computed is for the decomposition to be the described nontrivial extension.)
\end{proof}

\begin{lemma}\label{torconnlem} 
Suppose $X$ is a free $C_2$-surface. If there is a surface $Y$ such that $X\cong T \#_2 Y$ where $T$ is a free $C_2$-torus, then 
	\[
	H^{*,*}(X) \cong H^{*,*}(T) \oplus \left(\Sigma^{1,0}A_0\right)^{\oplus \beta (Y)}.
	\]
\end{lemma}
\begin{proof} 
We use the cofiber sequence
\begin{equation*}
	\left(Y'\times C_2\right)_+ \hookrightarrow \left(T \#_2 Y\right)_+ \to \tilde{T}
\end{equation*}
where $\tilde{T}$ is the space appearing in the cofiber sequence
\begin{equation*}
	C_{2+} \hookrightarrow T_{+} \to \tilde{T}.
\end{equation*}
The proof will follow similarly to the proof of Lemma \ref{anticonnlem}. We leave the details to the reader. 
\end{proof}
Let $A_n$ denote the cohomology of antipodal $n$-sphere, i.e. $A_n=\tau^{-1}\M_2/(\rho^{n+1})$. We can summarize the results of these lemmas in the following theorem. 
\begin{theorem}\label{freeanswer} 
Let $X$ be a free $C_2$-surface. Then there are two cases for the $RO(C_2)$-graded Bredon cohomology of $X$ in \underline{$\Z/2$}-coefficients.
\begin{enumerate}
	\item[(i)] Suppose $X$ is equivariantly isomorphic to $S^2_a\#_2Y$. Then
		\[
		H^{\ast, \ast}(X; \underline{\Z/2})\cong  \left( \Sigma^{1,0}A_0 \right)^{\oplus \beta(X)/2}\oplus A_2 
		\]
	\item[(ii)] Suppose $X$ is equivariantly isomorphic to $T\#_2 Y$ where $T$ is a free $C_2$-torus. Then
		\[
		H^{\ast, \ast}(X; \underline{\Z/2})\cong  \left( \Sigma^{1,0}A_0 \right)^{\oplus \frac{\beta(X)-2}{2}}\oplus A_1 \oplus \Sigma^{1,0}A_1 
		\]
\end{enumerate}
\end{theorem}
\begin{proof} 
This follows almost immediately from Lemma \ref{anticonnlem} and Lemma \ref{torconnlem}. It remains to check we have the appropriate number of summands of $\Sigma^{1,0}A_0$ in each case. Since $X$ is a free $C_2$-surface, $\chi(X)=2\chi(X/C_2)$ and so $\beta(X)=2\beta(X/C_2)-2$. Observe in (i), $X/C_2 \cong Y \# \R P^2$ so
	\[
	\beta(X) = 2\beta(X/C_2)-2 = 2(\beta(Y)+1)-2 = 2\beta(Y).
	\]
In (ii), $X/C_2\cong Y\# T_1$ or $X/C_2\cong Y \# (\R P^2 \#\R P^2)$ depending on the action on $T$. In either case,
	\[
	\beta(X) = 2\beta(X/C_2)-2 = 2(\beta(Y)+2)-2 = 2\beta(Y)+2.
	\]
In both (i) and (ii), solving for $\beta(Y)$ yields the desired number of summands.
\end{proof}

We have now completed the computation for free $C_2$-surfaces in the given coefficient system. The next section handles nonfree $C_2$-surfaces. 
%
%
\section{Computation for Nonfree Actions}\label{ch:nonfreecomp}
In this section, we compute the cohomology of all nonfree $C_2$-surfaces in coefficients given by the constant Mackey functor \underline{$\Z/2$}. We first prove some lemmas about how the various equivariant surgeries discussed in Section \ref{ch:comptools} affect the cohomology of a $C_2$-surface. Utilizing Theorem \ref{nonfreeclass} which states that all $C_2$-surfaces can be realized by doing such surgery to a simpler space, we then prove Theorem \ref{intro1}.
\begin{notation} In this section, all coefficients are understood to be \underline{$\Z/2$}, unless stated otherwise. Given a $C_2$-surface $X$, we will often use $F$ and $C$ to denote the number of isolated fixed points and the number of fixed circles, respectively. Whenever there is some ambiguity, such as when we are working with multiple $C_2$-surfaces at once, we will write $F(X)$ and $C(X)$ for the corresponding values.
\end{notation}
\begin{lemma}\label{freeplus11} 
Let $X$ be a $C_2$-surface. Suppose $X$ is isomorphic to $Y + [S^{1,1}-\text{\emph{AT}}]$ for some free $C_2$-surface $Y$. Then
	\[
	\tilde{H}^{*,*}(X;\underline{\Z/2}) \cong \left(\Sigma^{1,0} A_0\right)^{\frac{\beta(Y)+2}{2}}\oplus \Sigma^{2,2}\M_2.
	\]
\end{lemma}
\begin{proof}
Suppose $Y$ is a free $C_2$-surface. In Section \ref{ch:freecomp} we computed the cohomology of all such surfaces. Since $\tau$ acts invertibly on the modules $A_n$, we can regard these modules as $\tau^{-1}\M_2\cong \mathbb{F}_2[\tau,\tau^{-1},\rho]$-modules. If we ignore the action of $\rho$ (in other words, regard as just $\mathbb{F}_2[\tau,\tau^{-1}]$-modules), observe the cohomology of all free $C_2$-surfaces can be described as
	\[
	H^{*,*}(Y) \cong A_0\oplus \left(\Sigma^{1,0}A_0\right)^{\oplus \frac{\beta(Y)+2}{2}}\oplus \Sigma^{2,0}A_0
	\]
where as before $A_0\cong \tau^{-1}\M_2/(\rho)$. We now consider the following cofiber sequence for $X$
	\[
	S^{1,1} \hookrightarrow X \to \tilde{Y}
	\]
where $\tilde{Y}$ is the pinched space appearing in the cofiber sequence
	\[
	C_{2+} \hookrightarrow Y_+ \to \tilde{Y}.
	\]
We can compute the cohomology of $\tilde{Y}$ by extending to the cofiber sequence
\begin{equation}
	Y_+\to\tilde{Y}\to \Sigma^{1,0}C_{2+}.
\end{equation}
For this computation, it suffices to understand $\tilde{H}^{*,*}(\tilde{Y})$ as a $\mathbb{F}_2[\tau,\tau^{-1}]$-module. The differential is shown in Figure \ref{fig:tildey}.
\begin{figure}[ht]
	\begin{tikzpicture}[scale=.45]
		\draw[help lines,light-gray] (-5.125,-5.125) grid (5.125, 5.125);
		\draw[<->] (-5,0)--(5,0)node[right]{$p$};
		\draw[<->] (0,-5)--(0,5)node[above]{$q$};
		\anti{0}{0}{red};
		\anti{1.2}{0}{red};
		\anti{2}{0}{red};
		\anti{.8}{0}{blue};
		\lab{1.2}{\frac{\beta(Y)+2}{2}}{red};
		\draw[->](.5,.5)--(1.3,.5);
	\end{tikzpicture}\hspace{.5in}
	\begin{tikzpicture}[scale=.45]
		\draw[help lines,light-gray] (-5.125,-5.125) grid (5.125, 5.125);
		\draw[<->] (-5,0)--(5,0)node[right]{$p$};
		\draw[<->] (0,-5)--(0,5)node[above]{$q$};
		\anti{1.2}{0}{red};
		\anti{2}{0}{red};
		\lab{1.2}{\frac{\beta(Y)+2}{2}}{red};
	\end{tikzpicture}
	\caption{The differential ${\color{red}{\tilde{H}^{*,*}(Y_+)}}\to{\color{blue}{\tilde{H}^{*,*}(\Sigma^{1,0}C_{2+})}}$.}
	\label{fig:tildey}
\end{figure}

Now $\tilde{H}^{0,0}(\tilde{Y})\cong \tilde{H}^{0}_{sing}(\tilde{Y}/C_2)$ by Lemma \ref{quotient}, and $\tilde{H}^{0}_{sing}(\tilde{Y}/C_2)=0$. Hence $d^{0,0}$ must be an isomorphism, and by the module structure, $d^{0,q}$ must be an isomorphism for all $q$. We conclude 
	\[
	\tilde{H}^{*,*}(\tilde{Y})\cong \left(\Sigma^{1,0}A_0\right)^{\oplus \frac{\beta(Y)+2}{2}}\oplus \Sigma^{2,0}A_0
	\]
as $\mathbb{F}_2[\tau,\tau^{-1}]$-modules.

We now return to the cofiber sequence $S^{1,1} \hookrightarrow X \to \tilde{Y}$. The differential is shown in Figure \ref{fig:freeplus11}, where we have omitted possible $\rho$ actions in the picture of $\tilde{H}^{*,*}(\tilde{Y})$.
\begin{figure}[ht]
	\begin{tikzpicture}[scale=.45]
		\draw[help lines,light-gray] (-5.125,-5.125) grid (5.125, 5.125);
		\draw[<->] (-5,0)--(5,0)node[right]{$p$};
		\draw[<->] (0,-5)--(0,5)node[above]{$q$};
		\cone{1}{1}{red};
		\anti{.8}{0}{blue};
		\lab{.8}{\frac{\beta(Y)+2}{2}}{blue};
		\anti{2}{0}{blue};
		\draw[->](1.5,1.5)--(2.5,1.5);
	\end{tikzpicture}\hspace{.5 in}
	\begin{tikzpicture}[scale=.45]
		\draw[help lines,light-gray] (-5.125,-5.125) grid (5.125, 5.125);
		\draw[<->] (-5,0)--(5,0)node[right]{$p$};
		\draw[<->] (0,-5)--(0,5)node[above]{$q$};
		\anti{.8}{0}{blue};
		\lab{.8}{\frac{\beta(Y)+2}{2}}{blue};
		\draw[red,thick](1.5,-5)--(1.5,-.5)--(-3,-5);
		\draw[red,thick](2.5,5)--(2.5,2.5)--(5,5);
		\draw[blue,thick](2.5,.5)--(2.5,-5);
	\end{tikzpicture}
	\caption{The differential ${\color{red}{\tilde{H}^{*,*}(S^{1,1})}}\to{\color{blue}{\tilde{H}^{*+1,*}(\tilde{Y})}}$.}
	\label{fig:freeplus11}
\end{figure}
The differential must be nonzero for if it were zero, the answer for the cohomology of $X$ would not have a free summand generated in dimension $(2,2)$ as is required by Theorem \ref{topm2}. Using the module structure, we thus know $d^{p,q}$ for all $(p,q)$, and it remains to solve the extension problem given to the right in Figure \ref{fig:freeplus11}. We can conclude by Theorem \ref{topm2} (or by the forgetful long exact sequence) that the extension must be nontrivial, and we arrive at the desired answer. 
\end{proof}

\begin{lemma}\label{freeplus10}
Let $X$ be a $C_2$-surface that is isomorphic to $Y+[S^{1,0}-\text{\emph{AT}}]$ for some free $C_2$-surface $Y$. Then 
	\[
	\tilde{H}^{*,*}(X;\underline{\Z/2}) \cong \left(\Sigma^{1,0} A_0\right)^{\frac{\beta(Y)+2}{2}}\oplus \Sigma^{2,1}\M_2.
	\]
\end{lemma}
\begin{proof} 
The proof will follow as in the proof of Lemma \ref{freeplus11} except now using the cofiber sequence $S^{1,0}\hookrightarrow X \to \tilde{Y}$. We leave the details to the reader. 
\end{proof}

The following two lemmas state the cohomology of doubling spaces. The definition of these spaces can be found in Remark \ref{specialconnsum}

\begin{lemma}\label{doubspace11} 
Suppose $X$ is homeomorphic to \emph{Doub}$(Y,1:S^{1,1})$ for some nonequivariant surface $Y$. Then
	\[
	\H^{*,*}(X;\underline{\Z/2})\cong \left(\Sigma^{1,0}A_0\right)^{\beta(Y)} \oplus \Sigma^{2,2}\M_2.
	\]
\end{lemma}
\begin{proof} 
Consider the cofiber sequence
	\[
	S^{1,1}\hookrightarrow X \to C_{2+}\wedge Y.
	\]
By Lemma \ref{times} the cohomology of $C_{2+}\wedge Y$ is given by
	\[
	\tilde{H}^{p,q}(C_{2+}\wedge Y) \cong \Z/2[\tau,\tau^{-1}]\otimes_{\Z/2} \H^{*}_{sing}(Y)\cong  \left(\Sigma^{1,0} A_0 \right)^{\oplus \beta(Y)} \oplus \Sigma^{2,0} A_0.
	\]
The picture of the differential in the cofiber sequence is shown on the left in Figure \ref{fig:doubcomp}. The differential cannot be zero for if it were, the answer for the cohomology of $X$ would violate Theorem \ref{topm2}. Thus $d^{1,1}$ must be an isomorphism, and we can use the module structure to determine $d^{p,q}$ for all $p,q$. 
\begin{figure}[ht]
	\begin{tikzpicture}[scale=.45]
		\draw[help lines,light-gray] (-5.125,-5.125) grid (5.125, 5.125);
		\draw[<->] (-5,0)--(5,0)node[right]{$p$};
		\draw[<->] (0,-5)--(0,5)node[above]{$q$};
		\cone{1}{1}{red};
		\anti{.8}{0}{blue};
		\lab{.8}{\beta(Y)}{blue};
		\anti{2}{0}{blue};
		\draw[->](1.5,1.5)--(2.5,1.5);
	\end{tikzpicture}\hspace{.5 in}
	\begin{tikzpicture}[scale=.45]
		\draw[help lines,light-gray] (-5.125,-5.125) grid (5.125, 5.125);
		\draw[<->] (-5,0)--(5,0)node[right]{$p$};
		\draw[<->] (0,-5)--(0,5)node[above]{$q$};
		\anti{.8}{0}{blue};
		\lab{.8}{\beta(Y)}{blue};
		\draw[red,thick](1.5,-5)--(1.5,-.5)--(-3,-5);
		\draw[red,thick](2.5,5)--(2.5,2.5)--(5,5);
		\draw[blue,thick](2.5,.5)--(2.5,-5);
	\end{tikzpicture}
	\caption{The differential ${\color{red}{\tilde{H}^{*,*}(S^{1,1})}}\to{\color{blue}{\tilde{H}^{*+1,*}(C_{2+}\wedge Y)}}$.}
	\label{fig:doubcomp}
\end{figure}

It remains to solve the extension problem 
	\[
	0\to {\color{blue}{\coker(d)}} \to \tilde{H}^{*,*}(X) \to {\color{red}{\ker(d)}} \to 0	
	\]
which is shown on the right in Figure \ref{fig:doubcomp}; note the kernel is shown in red, and the cokernel is shown in blue. Using either Theorem \ref{topm2} or the forgetful long exact sequence, we can see the extension must be nontrivial, and in particular, 
	\[
	\tilde{H}^{*,*}(X;\underline{\Z/2}) \cong \left(\Sigma^{1,0} A_0\right)^{\beta(Y)}\oplus \Sigma^{2,2}\M_2.
	\]
\end{proof}

\begin{lemma}\label{doubspace10} 
Suppose $X$ is isomorphic to \emph{Doub}$(Y,1:S^{1,0})$ for some nonequivariant surface $Y$. Then
	\[
	\H^{*,*}(X;\underline{\Z/2})\cong \left(\Sigma^{1,0}A_0\right)^{\beta(Y)} \oplus \Sigma^{2,1}\M_2.
	\]
\end{lemma}
\begin{proof} The proof follows similarly to the proof of Lemma \ref{doubspace11} using the cofiber sequence
	\[
	S^{1,0}\hookrightarrow X \to C_{2+} \wedge Y.
	\]
We leave the details to the reader. 
\end{proof}

We are now ready to prove the main theorem about the cohomology of nonfree, nontrivial $C_2$-surfaces. The proof will make use of the classification given in Theorem \ref{nonfreeclass}. We will induct on the $\beta$-genus of the surface and explore the various cases for what type of surgery is needed to construct a given $C_2$-surface from a $C_2$-surface of lower $\beta$-genus.

\begin{theorem}\label{nonfreeanswer} 
Let $X$ be a nontrivial, nonfree $C_2$-surface. Let $F$ denote the number of isolated fixed points, $C$ denote the number of fixed circles, and $\beta$ denote the $\beta$-dimension. There are two cases for the reduced $RO(C_2)$-graded Bredon cohomology of $X$ in \underline{$\Z/2$}-coefficients:
\begin{enumerate}
	\item[(i)] Suppose $C=0$. Then
		\[
		\tilde H^{\ast, \ast}(X; \underline{\Z/2})\cong  \left(\Sigma^{1,1} \M_2\right)^{\oplus F-2} \oplus \left(\Sigma^{1,0}A_0 \right)^{\oplus \frac{\beta+2-F}{2}} \oplus \Sigma^{2,2}\M_2
		\]
	\item[(ii)] Suppose $C\neq 0$. Then 
	\begin{align*}
		\tilde H^{\ast, \ast}(X; \underline{\Z/2})\cong &  \left(\Sigma^{1,0} \M_2\right)^{\oplus C-1}\oplus  \left(\Sigma^{1,1} \M_2\right)^{\oplus F+C-1}\\
		&\oplus \left(\Sigma^{1,0}A_0 \right)^{\oplus \frac{\beta+2-(F+2C)}{2}} \oplus \Sigma^{2,1}\M_2
	\end{align*}
\end{enumerate}
\end{theorem}
\begin{proof} 
We begin with case (i) and proceed by induction on the $\beta$-genus of $X$. If $\beta(X)=0$, then $X$ must be an equivariant sphere. The only equivariant sphere with isolated fixed points is $S^{2,2}$. By the suspension theorem, the reduced cohomology of $S^{2,2}$ is $\Sigma^{2,2}\M_2$, which matches the decomposition given in (i).

Now suppose $\beta(X)>0$. By Theorem \ref{nonfreeclass}, we know $X$ is either a doubling space, or $X$ can be obtained by doing $S^{1,0}-$, $S^{1,1}-$, or $FM-$surgery to a $C_2$-surface of lower $\beta$-genus. If $X$ is a doubling space, then we are done by Lemma \ref{doubspace11}. If $X$ is obtained by doing equivariant surgery, then $X$ must be isomorphic to $Y+[S^{1,1}-\text{AT}]$ for some $Y$ by consideration of the fixed set. If $Y$ is free, then by Lemma \ref{freeplus11} the cohomology of $X$ is given by 
	\[
	\tilde{H}^{*,*}(X;\underline{\Z/2}) \cong \left(\Sigma^{1,0} A_0\right)^{\frac{\beta(Y)+2}{2}}\oplus \Sigma^{2,2}\M_2.
	\]
We just need to check there are the appropriate number of summands of each module. Observe $F(X)=2$ and $\beta(X) = \beta(Y)+2$, so indeed
	\[
	F(X)-2=0, \text{ and} ~\tfrac{\beta(X)+2-F(X)}{2} = \tfrac{\beta(Y)+2}{2}
	\]
as desired. \smallskip

If $Y$ is nonfree, then by consideration of the fixed set of $X$, $F(Y)\not = 0$ while $C(Y)=0$. Thus by induction, the cohomology of $Y$ is given by
\begin{equation}\label{caseiY}
	\tilde{H}^{\ast, \ast}(Y)\cong  \left(\Sigma^{1,1} \M_2\right)^{\oplus F(Y)-2} \oplus \left(\Sigma^{1,0}A_0 \right)^{\oplus \frac{\beta(Y)+2-F(Y)}{2}} \oplus \Sigma^{2,2}\M_2.
\end{equation}
Consider the cofiber sequence \[S^{1,1}\hookrightarrow X \to \tilde Y,\] where as before $\tilde Y$ is the space appearing in the cofiber sequence 
\begin{equation}\label{eq:ytilde}
	C_{2+}\hookrightarrow Y_+ \to \tilde{Y}.
\end{equation} 
Since $Y$ has at least one fixed point, we claim $\tilde{Y}$ is homotopy equivalent to $Y \vee S^{1,1}$. To see why, let's be more careful with how we construct $\tilde{Y}$. Let $y\in Y^{C_2}$ be a chosen fixed point. Since $Y$ has only isolated fixed points, there is a disk $D$ in $Y$ such that $y\in D$ and $D\cong D(\R^{2,2})$. Let $x\in D$ be an interior point that is not fixed, and include $C_2$ into $Y$ as $\{x,\sigma x\}$. Let $\gamma$ be a path from $x$ to $y$ contained in the interior of $D$ such that when we quotient to $\tilde{Y}$ 
	\[
	\im(\gamma) \cup \im(\sigma\gamma) \cong S^{1,1}.
	\]
See Figure \ref{fig:dtilde} for an illustration of the image of $\gamma$ in $\tilde{Y}$. 

There is a homotopy from $\tilde{D}=\cof(C_2\hookrightarrow D)$ to $D\vee S^{1,1}$ that keeps the boundary of $D$ fixed, and thus can be extended to all of $\tilde{Y}$ to see $\tilde{Y}\simeq Y\vee S^{1,1}$. See the illustration below of the homotopy equivalence for $\tilde{D}$. Note the fixed set is in blue and the copy of $S^{1,1}$ is shown in red. 
\begin{figure}[ht]
	\includegraphics[scale=0.8]{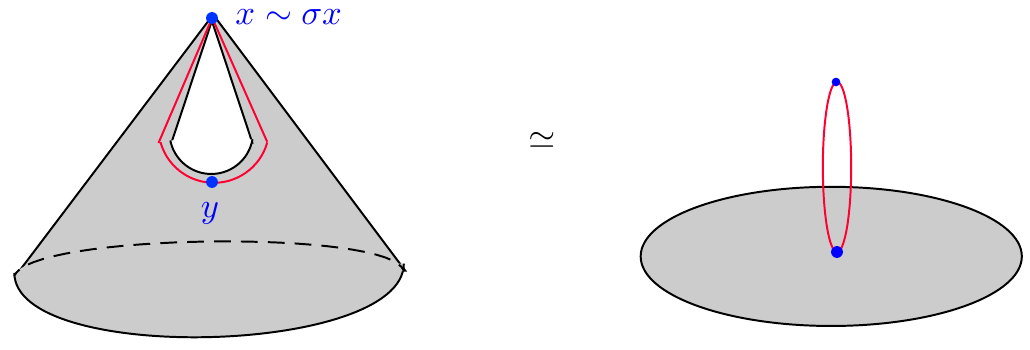}
	\caption{$\tilde{D} \simeq D\vee S^{1,1}$}
	\label{fig:dtilde}
\end{figure}

Using the homotopy discussed above, we see that
	\[
	\H^{*,*}(\tilde{Y}) \cong \H^{*,*}(Y) \oplus \Sigma^{1,1}\M_2.
	\]
The differential associated to the cofiber sequence in \ref{eq:ytilde} is shown below in Figure \ref{fig:caseicomp}. The number of summands of each of the blue submodules is omitted (note, in particular, the number of summands of $\Sigma^{1,0}A_0$ could be $0$). 
\begin{figure}[ht]
	\begin{tikzpicture}[scale=.45]
		\draw[help lines,light-gray] (-5.125,-5.125) grid (5.125, 5.125);
		\draw[<->] (-5.25,0)--(5.25,0)node[right]{$p$};
		\draw[<->] (0,-5.25)--(0,5.25)node[above]{$q$};
		\cone{2.2}{2}{blue};
		\cone{.8}{1}{blue};
		\anti{.7}{0}{blue};
		\cone{1}{1}{red};
		\draw[->] (1.5,1.5)--(2.5,1.5);
	\end{tikzpicture}
	\caption{The differential ${\color{red}{\tilde{H}^{*,*}(S^{1,1})}}\to{\color{blue}{\tilde{H}^{*,*}(\tilde{Y})}}$.}
	\label{fig:caseicomp}
\end{figure}
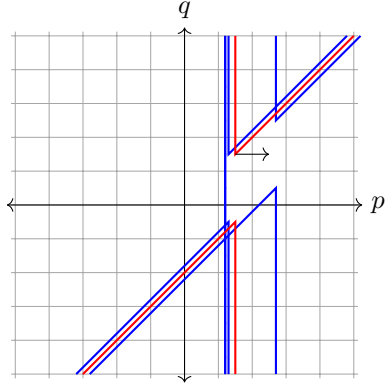
The differential must be zero because the generator of $\Sigma^{1,1}\M_2$ maps to the trivial group. Hence we must solve the extension problem
	\[
	0\to {\color{blue}{\tilde{H}^{*,*}(\tilde{Y})}} \to\tilde{H}^{*,*}(X) \to {\color{red}{\tilde{H}^{*,*}(S^{1,1})}}\to 0.  
	\]
The above splits since $\tilde{H}^{*,*}(S^{1,1})$ is a free $\M_2$-module, and so
\begin{align*}
	\tilde{H}^{*,*}(X) &\cong \tilde{H}^{*,*}(\tilde{Y}) \oplus \tilde{H}^{*,*}(S^{1,1})\\
	&\cong \tilde{H}^{*,*}(Y) \oplus \left(\Sigma^{1,1}\M_2\right)^{\oplus 2}.
\end{align*}
Putting this together with the isomorphism in \ref{caseiY}, the cohomology of $X$ is given by 
\begin{equation}\label{caseiX}
	\tilde{H}^{*,*}(X)\cong \left(\Sigma^{1,1} \M_2\right)^{\oplus F(Y)-2+2} \oplus \left(\Sigma^{1,0}A_0 \right)^{\oplus \frac{\beta(Y)+2-F(Y)}{2}} \oplus \Sigma^{2,2}\M_2.
\end{equation}
It remains to check there are the appropriate number of summands of each module. Adding an $S^{1,1}-$antitube increases both the number of isolated fixed points and the $\beta$-genus by $2$, which can be written as
\begin{align*}
	F(X)=F(Y)+2, \text{ and}~~~\beta(X)=\beta(Y)+2.
\end{align*}
The number of $\left(\Sigma^{1,1} \M_2\right)$-summands and the number of $\left(\Sigma^{1,0}A_0\right)$-summands in \ref{caseiX} can thus be written as
\begin{align*}
	F(Y)-2+2&=F(X)-2, \text{ and}\\
	\tfrac{\beta(Y)+2-F(Y)}{2}& = \tfrac{\beta(X)-2+2-(F(X)-2)}{2} = \tfrac{\beta(X)+2-F(X)}{2},
\end{align*}
respectively. We have completed the proof for case (i).\medskip

We now consider the case when $C(X)\neq 0$. We again proceed by induction on the $\beta$-genus of $X$. If $\beta(X)=0$, then $X$ must be an equivariant sphere, and the only equivariant sphere with a fixed circle is $S^{2,1}$. By the suspension theorem, the reduced cohomology of $S^{2,1}$ is $\Sigma^{2,1}\M_2$, which matches the decomposition given in (ii).

Now suppose $\beta(X)>0$. If $X$ is a doubling space, we are done by Lemma \ref{doubspace10}. We can thus assume $X$ is obtained by doing equivariant surgery to a $C_2$-surface $Y$ of lower $\beta$-genus. Since we know $X^{C_2}$ contains at least one fixed circle, we can assume $X$ is obtained by doing doing $S^{1,0}-$ or $FM-$ surgery to an equivariant surface.

Let's first assume $X\cong Y +[S^{1,0}-AT]$. If $Y$ is a free $C_2$-surface, then we are done after applying Lemma \ref{freeplus10} and noting $C(X)=1$ and $\beta(X)=\beta(Y)+2$. Thus suppose $Y$ is a nonfree $C_2$-surface and consider the cofiber sequence
	\[
	S^{1,0}\hookrightarrow X \to \tilde{Y}.
	\]
Since $Y^{C_2}$ is nonempty, we can make a similar argument as before to see $\tilde{Y} \simeq Y \vee S^{1,1}$ and so
	\[
	\tilde{H}^{*,*}(\tilde{Y}) \cong \H^{*,*}(Y)\oplus \Sigma^{1,1}\M_2.
	\]
 If $Y^{C_2}$ contains at least one fixed circle, then we know the cohomology of $Y$ by induction. If $Y^{C_2}$ contains only isolated fixed points, then we know the cohomology of $Y$ by case (i). Thus there are two possibilities for the differential appearing in the long exact sequence associated to the cofiber sequence above; both are shown below in Figure \ref{fig:caseiicomp1}. 
\begin{figure}[ht]
	\begin{tikzpicture}[scale=.45]
		\draw[help lines,light-gray] (-5.125,-5.125) grid (5.125, 5.125);
		\draw[<->] (-5.25,0)--(5.25,0)node[right]{$p$};
		\draw[<->] (0,-5.25)--(0,5.25)node[above]{$q$};
		\draw[thick,blue] (2.7,5)--(2.7,1.5)--(5,3.8);
		\draw[thick,blue] (2.7,-5)--(2.7,-.5)--(-2.2,-5);
		\cone{.9}{1}{blue};
		\draw[thick,blue] (1.3,5)--(1.3,.5)--(5,4.2);
		\draw[thick,blue] (1.3,-5)--(1.3,-1.5)--(-2.4,-5);
		\anti{.7}{0}{blue};
		\draw[thick,red] (1.6,5)--(1.6,.5)--(5,3.9);
		\draw[thick,red] (1.6,-5)--(1.6,-1.5)--(-1.9,-5);
		\draw[->] (1.6,0.5)--(2.5,0.5);
	\end{tikzpicture}\hspace{0.5in}
	\begin{tikzpicture}[scale=.45]
		\draw[help lines,light-gray] (-5.125,-5.125) grid (5.125, 5.125);
		\draw[<->] (-5.25,0)--(5.25,0)node[right]{$p$};
		\draw[<->] (0,-5.25)--(0,5.25)node[above]{$q$};
		\cone{2.2}{2}{blue};
		\cone{.9}{1}{blue};
		\draw[thick,blue] (1.3,5)--(1.3,.5)--(5,4.2);
		\draw[thick,blue] (1.3,-5)--(1.3,-1.5)--(-2.4,-5);
		\anti{.7}{0}{blue};
		\draw[thick,red] (1.6,5)--(1.6,.5)--(5,3.9);
		\draw[thick,red] (1.6,-5)--(1.6,-1.5)--(-1.9,-5);
		\draw[->] (1.6,0.5)--(2.7,0.5);
	\end{tikzpicture}
	\caption{The two cases for the differential ${\color{red}{\tilde{H}^{*,*}(S^{1,0})}}\to{\color{blue}{\tilde{H}^{*,*}(\tilde{Y})}}$.}
	\label{fig:caseiicomp1}
\end{figure}
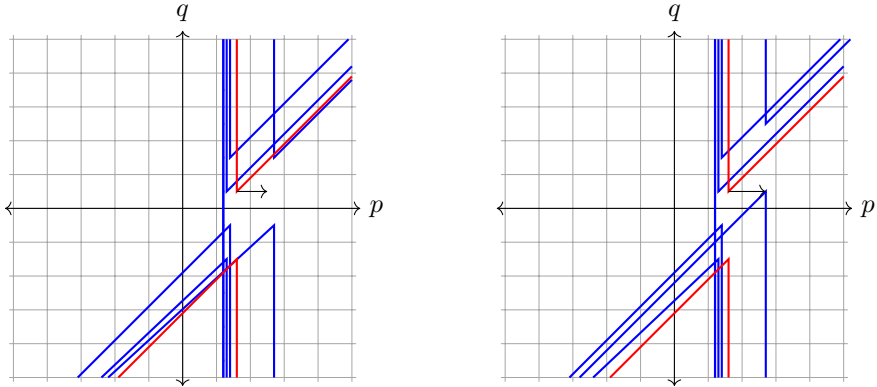
The case illustrated on the left is when $Y^{C_2}$ contains at least one fixed circle. In this case, we immediately see the differential must be $0$ and conclude the extension is trivial because the kernel is a free $\M_2$-module. Thus 
	\[
	\tilde{H}^{*,*}(X) \cong \tilde{H}^{*,*}(Y) \oplus \Sigma^{1,1}\M_2 \oplus \Sigma^{1,0}\M_2,
	\]
which by induction gives
\begin{align*}
	\tilde{H}^{*,*}(X) \cong  &\Big(\left(\Sigma^{1,0} \M_2\right)^{\oplus C(Y)-1}\oplus  \left(\Sigma^{1,1} \M_2\right)^{\oplus F(Y)+C(Y)-1}\\
	& \oplus \left(\Sigma^{1,0}A_0 \right)^{\oplus \frac{\beta(Y)+2-(F(Y)+2C(Y))}{2}} \oplus \Sigma^{2,1}\M_2\Big) \oplus \Sigma^{1,1}\M_2 \oplus \Sigma^{1,0}\M_2\\
	\cong&\left(\Sigma^{1,0} \M_2\right)^{\oplus C(Y)}\oplus  \left(\Sigma^{1,1} \M_2\right)^{\oplus F(Y)+C(Y)}\\
	& \oplus \left(\Sigma^{1,0}A_0 \right)^{\oplus \frac{\beta(Y)+2-(F(Y)+2C(Y))}{2}} \oplus \Sigma^{2,1}\M_2.
\end{align*}
Recall $X\cong Y+[S^{1,0}-AT]$, and so 
	\[
	F(Y)=F(X),~C(Y)=C(X)-1,\text{ and}~\beta(Y)=\beta(X)-2.
	\]
By making the above substitutions, we arrive at the desired answer for case (ii).

The case on the right in Figure \ref{fig:caseiicomp1} is slightly more complicated. The differential cannot be zero in this case for if it were, the answer for the cohomology of $X$ would violate Theorem \ref{topm2}. Noting $d^{1,0}$ must be nonzero and using the module structure to determine $d^{p,q}$ for all $(p,q)$, we solve the extension problem
	\[
	0 \to {\color{blue}{\coker(d)}} \to \tilde{H}^{*,*}(X) \to {\color{red}{\ker(d)}} \to 0
	\]
which is illustrated below in Figure \ref{fig:extprobii}. The kernel of $d$ is shown in red while the cokernel is shown in blue. 

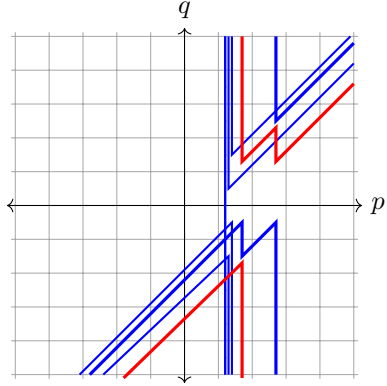
\begin{figure}[ht]
	\begin{tikzpicture}[scale=.45]
		\draw[help lines,light-gray] (-5.125,-5.125) grid (5.125, 5.125);
		\draw[<->] (-5.25,0)--(5.25,0)node[right]{$p$};
		\draw[<->] (0,-5.25)--(0,5.25)node[above]{$q$};
		\cone{.9}{1}{blue};
		\draw[thick,blue] (1.3,5)--(1.3,.5)--(5,4.2);
		\draw[thick,blue] (1.3,-5)--(1.3,-1.5)--(-2.4,-5);
		\anti{.7}{0}{blue};
		\draw[very thick,blue] (2.7,5)--(2.7,2.5)--(5,4.8);
		\draw[very thick,blue] (2.7,-5)--(2.7,-.5)--(1.7,-1.5)--(1.7,-.5)--(-2.8,-5);
		\draw[very thick,red] (1.7,-5.1)--(1.7,-1.7)--(-1.8,-5.1);
		\draw[very thick,red] (1.7,5)--(1.7,1.3)--(2.7,2.3)--(2.7,1.3)--(5,3.6);
	\end{tikzpicture}
	\caption{The extension problem.}
	\label{fig:extprobii}
\end{figure}

From Theorem \ref{topm2}, we have that $\Sigma^{2,1}\M_2$ must be a summand of $\tilde{H}^{*,*}(X)$, and so the extension is nontrivial. In particular, the extension must be given by
\begin{align*}
	\tilde{H}^{*,*}(X) \cong \left(\Sigma^{1,1} \M_2\right)^{\oplus F(Y)-2} \oplus \left(\Sigma^{1,0}A_0 \right)^{\oplus \frac{\beta(Y)+2-F(Y)}{2}} \oplus \Sigma^{1,1}\M_2 \oplus \Sigma^{2,1}\M_2
\end{align*}
where the last two summands arise from the nontrivial extension. It remains to check there are the appropriate number of summands of each submodule. Recall $X\cong Y +[S^{1,0}-AT]$ where $Y^{C_2}$ consisted only of isolated fixed points. Thus
	\[
	F(Y)=F(X), ~C(X)=1,\text{ and}~\beta(Y)=\beta(X)-2.
	\]
Making the substitution for $F(Y)$ and $\beta(Y)$ and noting $C(X)-1=0$, we arrive at the desired decomposition given in (ii).\smallskip

The only remaining case is when $X\cong Y+[FM]$ for some $C_2$-surface $Y$ of lower $\beta$-genus. Let $D$ be an equivariant closed neighborhood containing the attached M\"obius band so that $D \simeq M \simeq S^{1,0}$. Consider the following cofiber sequence
	\[
	D \hookrightarrow Y+[FM] \to Y.
	\]
Notice $Y^{C_2}$ must be nonempty in order to do $FM-$surgery, so we know the cohomology of $Y$ either from induction or from part (i). There are two cases for the cohomology of $Y$ depending on whether or not $Y^{C_2}$ contains a fixed circle. Similar to when $X\cong Y+[S^{1,0}-AT]$, this yields two cases for the differentials appearing in the long exact sequences associated to the above cofiber sequence. Though, note the cohomology of $Y$ will now be appearing rather than the cohomology of $\tilde{Y}$. 

The two cases for the differential are shown in Figure \ref{fig:caseiicomp2} below. We only include the relevant portion of the cohomology of $Y$ in the picture, noting the other summands cannot be in the image of the differential for degree reasons as in Figure \ref{fig:caseiicomp1}.
\begin{figure}[ht]
	\begin{tikzpicture}[scale=.45]
		\draw[help lines,light-gray] (-5.125,-5.125) grid (5.125, 5.125);
		\draw[<->] (-5,0)--(5,0)node[right]{$p$};
		\draw[<->] (0,-5)--(0,5)node[above]{$q$};
		\draw[thick,blue] (2.4,5)--(2.4,1.5)--(5,4.1);
		\draw[thick,blue] (2.4,-5)--(2.4,-.5)--(-2.1,-5);
		\draw[thick,red] (1.6,5)--(1.6,.5)--(5,3.9);
		\draw[thick,red] (1.6,-5)--(1.6,-1.5)--(-1.9,-5);
		\draw[->](1.6,.5)--(2.6,.5);
	\end{tikzpicture}\hspace{.5in}
	\begin{tikzpicture}[scale=.45]
		\draw[help lines,light-gray] (-5.125,-5.125) grid (5.125, 5.125);
		\draw[<->] (-5,0)--(5,0)node[right]{$p$};
		\draw[<->] (0,-5)--(0,5)node[above]{$q$};
		\cone{2}{2}{blue};
		\draw[thick,red] (1.6,5)--(1.6,.5)--(5,3.9);
		\draw[thick,red] (1.6,-5)--(1.6,-1.5)--(-1.9,-5);
		\draw[->](1.6,.5)--(2.5,.5);
	\end{tikzpicture}
	\caption{The two cases for the differential $d:{\color{red}{\H^{*,*}(S^{1,0})}}\to {\color{blue}{\H^{*,*}(Y)}}$.}
	\label{fig:caseiicomp2}
\end{figure}
The picture on the left shows the case when the fixed set of $Y$ contains at least one fixed circle. In this case, we quickly see the differential must be zero and the extension must be trivial. Thus
\begin{align*}
	\tilde{H}^{*,*}(X) \cong \tilde{H}^{*,*}(Y) \oplus \Sigma^{1,0}\M_2.
\end{align*}
By induction, we have
\begin{align*}
	\tilde{H}^{*,*}(X) \cong &  \Big(\left(\Sigma^{1,0} \M_2\right)^{\oplus C(Y)-1}\oplus  \left(\Sigma^{1,1} \M_2\right)^{\oplus F(Y)+C(Y)-1}\\
	&\oplus \left(\Sigma^{1,0}A_0 \right)^{\oplus \frac{\beta(Y)+2-(F(Y)+2C(Y))}{2}} \oplus \Sigma^{2,1}\M_2\Big) \oplus \Sigma^{1,0}\M_2\\
	\cong &   \left(\Sigma^{1,0} \M_2\right)^{\oplus C(Y)}\oplus  \left(\Sigma^{1,1} \M_2\right)^{\oplus F(Y)+C(Y)-1}\\
	&\oplus \left(\Sigma^{1,0}A_0 \right)^{\oplus \frac{\beta(Y)+2-(F(Y)+2C(Y))}{2}} \oplus \Sigma^{2,1}\M_2.
\end{align*}
Recall $X\cong Y+[FM]$, and so
	\[
	F(Y)=F(X)+1,~C(Y)=C(X)-1, \text{ and} ~ \beta(Y)=\beta(X)-1.
	\]
These substitutions will give the decomposition stated in (ii).

The final case is $X\cong Y+[FM]$ where $Y^{C_2}$ consists only of isolated fixed points. We return to the differential shown on the right in Figure \ref{fig:caseiicomp2}. Observe on the quotient level, doing $FM-$surgery removes a disk. Thus $X/C_2$ is $Y$ with a disk removed, so $H^{2}_{sing}(X/C_2)=0$. By Lemma \ref{quotient}, $H^{2,0}(X) \cong H^{2}_{sing}(X/C_2)$, and so we conclude $d^{1,0}$ must be an isomorphism. We now must solve the extension problem shown below. 
\begin{figure}[ht]
	\begin{tikzpicture}[scale=.45]
		\draw[help lines,light-gray] (-5.125,-5.125) grid (5.125, 5.125);
		\draw[<->] (-5,0)--(5,0)node[right]{$p$};
		\draw[<->] (0,-5)--(0,5)node[above]{$q$};
		\draw[blue,thick] (2.5,5)--(2.5,2.5)--(5,5);
		\draw[blue,thick] (2.5,-5)--(2.5,-.5)--(1.5,-1.5)--(1.5,-.5)--(-3,-5);
		\draw[red,thick] (1.5,5)--(1.5,1.4)--(2.5,2.4)--(2.5,1.4)--(5,3.9);
		\draw[red,thick] (1.5,-5)--(1.5,-1.6)--(-1.9,-5);
	\end{tikzpicture}
	\caption{The extension problem.}
\end{figure}
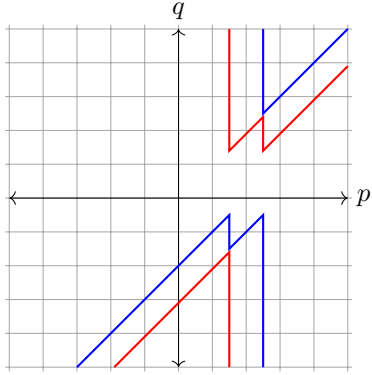
The extension problem is solved after applying Theorem \ref{topm2}, and we conclude 
\begin{align*}
	\tilde{H}^{*,*}(X) &\cong \left(\Sigma^{1,1} \M_2\right)^{\oplus F(Y)-2} \oplus \left(\Sigma^{1,0}A_0 \right)^{\oplus \frac{\beta(Y)+2-F(Y)}{2}} \oplus \Sigma^{1,1}\M_2 \oplus \Sigma^{2,1}\M_2\\
	&\cong  \left(\Sigma^{1,1} \M_2\right)^{\oplus F(Y)-1} \oplus \left(\Sigma^{1,0}A_0 \right)^{\oplus \frac{\beta(Y)+2-F(Y)}{2}} \oplus \Sigma^{2,1}\M_2.
\end{align*}
Observe $F(Y)=F(X)+1$, $C(X)=1$, and $\beta(Y)=\beta(X)-1$. Making these substitutions will finish the proof.
\end{proof}
%
%
\appendix
\section{Proof of Theorem \ref{topm2}}\label{ch:manfacts}
In this appendix we provide a proof of Theorem \ref{topm2} which is given as Theorem \ref{topm22} below. Here by ``manifold" we mean a piecewise linear manifold, and by $C_2$-action, we mean a locally linear $C_2$-action. Note this is sufficient to guarantee the fixed set is a disjoint union of submanifolds.

\begin{theorem}\label{topm22} 
Let $X$ be an $n$-dimensional, closed $C_2$-manifold with a nonfree $C_2$-action. Suppose $n-k$ is the largest dimension of submanifold appearing as a component of the fixed set. Then there is exactly one summand of $\H^{*,*}(X;\underline{\Z/2})$ of the form $\Sigma^{i,j}\M_2$ where $i\geq n$, and it occurs for $(i,j)=(n,k)$.
\end{theorem}
\begin{proof} 
If $X$ is a trivial space, then this follows immediately from Lemma \ref{trivial} and facts about the singular cohomology of closed $n$-manifolds in $\Z/2$-coefficients. Thus assume $X$ is nontrivial. We first show there is a unique summand generated in topological dimension $n$, we then show it must be free, and lastly we argue it must be in weight $k$.

From the structure theorem given in Theorem \ref{structure}, the cohomology of $X$ must have a direct sum decomposition given by

\begin{equation}\label{eq:decomp}
	H^{*,*}(X) \cong (\oplus_{i} \Sigma^{m_{i},k_{i}}\M_2) \oplus (\oplus_{j} \Sigma^{r_j}A_{s_j}).
\end{equation}
Consider the following portion of the forgetful long exact sequence for $X$:
	\[
	H^{p-1,q}(X)\overset{\rho}{\longrightarrow} H^{p,q+1}(X)\overset{\psi}{\longrightarrow} H^{p}_{sing}(X)\longrightarrow H^{p,q}(X)
	\]
Since $X$ is a closed $n$-manifold, $H^{p}_{sing}(X)=0$ for $p>n$ while $H^{n}_{sing}(X)\cong \Z/2$. By exactness, it must be that $H^{p,q}(X) = \im(\rho)$ for $p>n$, and when $p=n$, there are two possibilities: either $H^{n,q}(X) = \im(\rho)$ or $H^{n,q}(X)/\im(\rho)\cong \Z/2$. Returning to \ref{eq:decomp}, this immediately implies $m_i,r_j \leq n$ for all $i$, $j$. We claim this also implies there are either zero summands or one summand generated in topological dimension $n$; that is there is at most one $i$ or $j$ such that $m_i=n$ or $r_j=n$. Indeed, if there were two or more summands generated in topological dimension $n$, then there would exist a sufficiently large $q$ such that $\dim\left(H^{n,q}(X)/\im(\rho)\right)\geq 2$.

To see there is a summand generated in topological dimension $n$, pick a point $x\in X^{C_2}$ that is contained in a connected component of dimension $n-k$, where recall $n-k$ is the maximum dimension. There exists an open equivariant disk $D$ such that $x\in D\subset X$ and $D\cong \R^n$ nonequivariantly. By consideration of the fixed set, we see that $D\cong D(\R^{n,k})$ where $D(\R^{n,k})$ denotes the unit disk in $\R^{n,k}$. Consider the quotient map
	\[
	q:X\to X/(X-D) \cong S^{n,k}.
	\]

We have the following commutative square involving the forgetful map:
\begin{center}
	\begin{tikzcd}
	\H^{n,k}(S^{n,k}) \arrow[r,"\psi"]\arrow[d,"q^*"]& \H^{n}_{sing}(S^{n,k})\arrow[d,"q^*"]\\
	\H^{n,k}(X) \arrow[r,"\psi"]& \H^{n}_{sing}(X)
	\end{tikzcd}
\end{center}
Recall $\H^{*,*}(S^{n,k})\cong \Sigma^{n,k}\M_2$ by the suspension isomorphism, so the top map is an isomorphism. The right vertical map is also an isomorphism because $X$ is a closed $n$-manifold. By commutativity of the square, the forgetful map $\psi:\H^{n,k}(X) \to \H^{n}_{sing}(X)$ must be nonzero. Returning to the forgetful long exact sequence above, we see $H^{n,k}(X)/\im(\rho)\cong \Z/2$ and there is indeed exactly one summand generated in topological dimension $n$. \medskip

There are thus two options for the cohomology of $X$. Either
\begin{equation}\label{eq:case1}
	H^{*,*}(X) \cong (\oplus_{i} \Sigma^{m_{i},k_{i}}\M_2) \oplus (\oplus_{j} \Sigma^{r_j}A_{s_j}) \oplus \Sigma^{n}A_b; \text{ or}
\end{equation}
\begin{equation}\label{eq:case2}
	H^{*,*}(X) \cong (\oplus_{i} \Sigma^{m_{i},k_{i}}\M_2) \oplus (\oplus_{j} \Sigma^{r_j}A_{s_j})  \oplus \Sigma^{n,c}\M_2.
\end{equation}
where in both equations $m_i,r_j<n$. We show the first case cannot happen.

Suppose to the contrary the cohomology of $X$ is given by \ref{eq:case1}. Let $p$ be a nonfixed point and consider the punctured space $X-\{p,\sigma p\}\simeq X-\{D(p), \sigma D(p)\}$ where $D(p)$ is a small open disk around $p$ that does not intersect its conjugate disk $\sigma D(p)$. Note $X-\{D(p),\sigma D(p)\}$ is an $n$-manifold with boundary, so 
	\[
	H^{j}_{sing}(X-\{D(p),\sigma D(p)\})=0 \text{ for } j\geq n.
	\] 
We can put a $C_2$-CW structure on $X-\{D(p),\sigma D(p)\}$ with no cells of dimension greater than $n$. The map $q:X-\{D(p),\sigma D(p)\}\to (X-\{D(p),\sigma D(p)\})/C_2$ will be a cellular map that induces a levelwise surjective map on the cellular chain complexes. Using this map of chain complexes and the above fact, a diagram chase shows
	\[
	H^{j}_{sing}((X-\{D(p),\sigma D(p)\})/C_2)=0 \text{ for } j\geq n.
	\] 
By the quotient lemma given in \ref{quotient}, it follows that $H^{j,0}(X-\{p,\sigma p\})=0$ for $j\geq n$.

Consider the pair $(X,X-\{p,\sigma p\})$. Note
	\[
	H^{*,*}(X,X-\{p,\sigma p\}) \cong \H^{*,*}(X/(X-\{D(p),\sigma D(p)\})) \cong \H^{*,*}(C_{2+}\wedge S^n) \cong \Sigma^{n}A_0.
	\]
We have the following diagram where the rows are exact:
\begin{center}
	\begin{tikzcd}
	H^{n,0}(X,X-\{p,\sigma p\})\arrow[r] \arrow[d,"\psi"]&H^{n,0}(X) \arrow[r]\arrow[d,"\psi"] &H^{n,0}(X-\{p,\sigma p\})\arrow[d,"\psi"]\\
	H^{n}_{sing}(X,X-\{p,\sigma p\})\arrow[r] &H^{n}_{sing}(X) \arrow[r] &H^{n}_{sing}(X-\{p,\sigma p\})
	\end{tikzcd}
\end{center}
The right-hand groups are both zero by the above discussion, so the left horizontal maps are both surjective. The middle vertical map is surjective based on the decomposition given in \ref{eq:case1}, while the left vertical map is given by the diagonal map
	\[
	\H^{*,*}(C_{2+}\wedge S^n) \cong \H^*_{sing}(S^n)\to \H^{*}_{sing}(S^n\vee S^n)\cong \H^*_{sing}(S^n)\oplus \H^{*}_{sing}(S^n).
	\]
Thus we have the following commutative diagram coming from the left square where $\Delta$ is the diagonal map and $\nabla$ is the fold map.
\begin{center}
	\begin{tikzcd}
	\Z/2\arrow[r,twoheadrightarrow] \arrow[d,"\Delta"]&H^{n,0}(X) \arrow[d,"\psi", twoheadrightarrow] \\
	\Z/2\oplus \Z/2\arrow[r,"\nabla"] &\Z/2
	\end{tikzcd}
\end{center}
We have arrived at a contradiction: going around the diagram one way is zero, while the other way is nonzero. We conclude the cohomology of $X$ must have a decomposition as in \ref{eq:case2}. In particular, there is a unique free summand in topological dimension $n$ and furthermore there are no other summands generated in topological dimension greater than or equal to $n$. \medskip

We now show this free summand is generated in weight $k$. Let's reconsider the quotient map
	\[
	q:X\to X/(X-D) \cong S^{n,k}.
	\]
Let $s\geq k$. We have the following map between the forgetful long exact sequences for $X$ and $S^{n,k}$.
\begin{center}
	\begin{tikzcd}
	\dots\arrow[r]&\H^{n-1,s-1}(S^{n,k}) \arrow[r,"\rho"]\arrow[d,"q^*"] &\H^{n,s}(S^{n,k}) \arrow[r,"\psi"]\arrow[d,"q^*"]& \H^{n}_{sing}(S^{n,k})\arrow[d,"q^*","\cong" left]\arrow[r]&\dots\\
	\dots\arrow[r]&\H^{n-1,s-1}(X) \arrow[r,"\rho"] &\H^{n,s}(X) \arrow[r,"\psi"]& \H^{n}_{sing}(X)\arrow[r]&\dots
	\end{tikzcd}
\end{center}
Recall $\H^{*,*}(S^{n,k}) \cong \Sigma^{n,k}\M_2$ by the suspension isomorphism. We provide an illustration of this cohomology below for reference. 
\begin{figure}[ht]
	\begin{tikzpicture}[scale=.4]
		\draw[help lines,light-gray] (-2.125,-2.125) grid (7.125, 7.125);
		\draw[thick,<->] (-2,0)--(7,0)node[right]{$p$};
		\draw[thick,<->] (0,-2)--(0,7)node[above]{$q$};
		\draw[thick] (4,.2)--(4,-.2)node[below]{$n$};
		\draw[thick] (.2,3)--(-.2,3)node[left]{$k$};
		\draw[very thick] (4.5,7)--(4.5,3.5)--(7,6);
		\draw[very thick] (4.5,-2)--(4.5,1.5)--(1,-2);
\end{tikzpicture}
\caption{The reduced cohomology of $S^{n,k}$.}
\end{figure}
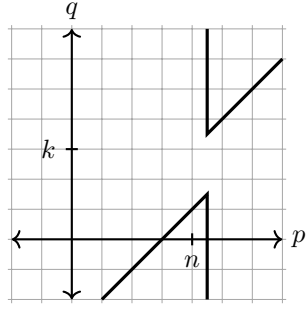

The above map of long exact sequences is thus
\begin{center}
	\begin{tikzcd}
	0 \arrow[r]\arrow[d,"q^*"] &\Z/2 \arrow[r,"\psi","\cong" below]\arrow[d,"q^*"]& \Z/2 \arrow[d,"q^*","\cong" left]\\
	\H^{n-1,s-1}(X) \arrow[r,"\rho"] &\H^{n,s}(X) \arrow[r,"\psi"]& \Z/2
	\end{tikzcd}
\end{center}
The square on the right shows 
	\[
	q^*:\H^{n,s}(S^{n,k}) \to \H^{n,s}(X)
	\]
is injective for all $s\geq k$. If we let $a\in \H^{n,k}(S^{n,k})$ be the generator, the exactness also shows the nonzero element $q^*a$ is not in the image of $\rho$.

Returning to decomposition given in \ref{eq:case2}, we can write $q^*a$ as an $\M_2$-combination of generators of the summands. Observe a $\tau$-multiple of the generator of the summand $\Sigma^{n,c}\M_2$ must appear in this linear combination since otherwise $q^*a$ would be in the image of $\rho$. Thus the weight $c$ must be less than or equal to $k$. 

To see $c=k$, note $\rho$-localization will yield a generator in $\H^{n-c}_{sing}(X^{C_2})$. Since $n-k$ is the largest dimension appearing in the fixed set, it must be that $n-c\leq n-k$ or $c\geq k$. We conclude $c=k$, as desired.
\end{proof}

We end by mentioning one corollary of the above proof.

\begin{corollary} Let $X$ be an $n$-dimensional nonfree $C_2$-manifold and let $x\in X^{C_2}$ be a point in a component of the fixed set of smallest codimension $k$. Then the map $q:X\to S^{n,k}$ that collapses the complement of a small disk around $x$ to a point induces a split injection.
\end{corollary}
\begin{proof} In the proof above, we showed $q^*a$ where $a$ is the generator of $\H^{*,*}(S^{n,k})$ generates a free summand of $H^{*,*}(X)$. This implies the map is injective, and it is split because $\M_2$ is self-injective. 
\end{proof}

\bibliographystyle{amsalpha}

\end{document}